\documentclass[a4paper,twoside,11pt]{article}

\usepackage{color}
\usepackage{appendix}

\usepackage[latin1]{inputenc}
\usepackage{lmodern}
\usepackage[T1]{fontenc}

\usepackage[a4paper,top=3cm,bottom=3cm,left=3cm,right=3cm,
bindingoffset=5mm]{geometry}

\usepackage{amsmath,amssymb,amsthm}
\usepackage{graphicx}

\usepackage{epstopdf}

\usepackage{subfigure}
\usepackage{epstopdf}

\newtheorem{thm}{Theorem}[section]
\newtheorem{prop}[thm]{Proposition}
\newtheorem{lem}[thm]{Lemma}

\newtheorem{remark}{Remark}

\renewcommand{\O}{\Omega}

\newcommand\E{{E}}  

\newcommand\Th{{\mathcal T}_h}

\newcommand\Eh{{\mathcal E}_h}

\newcommand\R{\mathbb{R}}
\newcommand\C{\mathbb{C}}
\newcommand\D{\mathbb{D}}

\renewcommand{\P}{{\mathcal P}}  

\def\decapita#1{}
\def\grecomultibold#1#2{\grecobolddef#1\def\secondobold{#2}%
    \ifx#2\finemultibold\let\next\relax\let\secondobold\relax
    \else\let\next\grecomultibold
    \fi\expandafter\next\secondobold}

\def\grecobolddef#1{%
  \edef\dadef{bf\expandafter\decapita\string#1}%
  \expandafter\def\csname\dadef\endcsname{{\neretto #1}}}

\def\neretto#1{\setbox0=\hbox{\mathsurround=0pt$#1$}%
  \kern.02em\copy0 \kern-\wd0
  \kern-.02em\copy0 \kern-\wd0
  \raise.03em\box0 \kern.02em}
\grecomultibold\alpha\gamma\delta\xi\sigma\tau\eta\theta\varepsilon\rho\pi\varphi\varepsilon\finemultibold
\def\bdiv{\mathop{\bf div}\nolimits}

\def\curl{\mathop{\rm curl}\nolimits}
\def\bcurl{\mathop{\bf curl}\nolimits}

\def\teps{{\bfvarepsilon}}

\def\A{{\mathcal A}}

\def\Ih{{\mathcal I}_h}
\def\IE{{\mathcal I}_E}

\def\wbox#1;#2;{\vbox{\hrule\hbox{\vrule height#1mm\kern#2mm\vrule
  height#1mm}\hrule}}

\let\phi\varphi

\let\b\b

\newcommand{\bba}   {\mathbf{a}}

\newcommand{\bbc}   {\mathbf{c}}

\newcommand{\bbf}   {\mathbf{f}}
\newcommand{\bbg}   {\mathbf{g}}
\newcommand{\bbh}   {\mathbf{h}}

\newcommand{\bbn}   {\mathbf{n}}

\newcommand{\bbp}   {\mathbf{p}}
\newcommand{\bbq}   {\mathbf{q}}
\newcommand{\bbr}   {\mathbf{r}}

\newcommand{\bbt}   {\mathbf{t}}
\newcommand{\bbu}   {\mathbf{u}}
\newcommand{\bbv}   {\mathbf{v}}
\newcommand{\bbw}   {\mathbf{w}}
\newcommand{\bbx}   {\mathbf{x}}

\def\C{\mathbb C}

\def\hpoint#1.#2.#3{{\underline{#1}}_{#2}\cdot
 {\underline{\mathop{\smash{#3}\vphantom{{#1}_{#2}}}}}}

\def\npointp#1.#2{{\underline{#1}}\cdot
 {\underline{\mathop{\smash{#2}\vphantom{{#1}}}}}}

\def\beq{\begin{equation}}
\def\enq{\end{equation}}

\def\bfzero{{\bf 0}}

\author{
E. Artioli\thanks{Department of Civil Engineering and Computer Science,
	University of Rome Tor Vergata,
	Via del Politecnico 1, 00133 Rome, Italy,
	{\tt artioli@ing.uniroma2.it}},
S. {d}e Miranda\thanks{DICAM, University of Bologna, Viale Risorgimento 2, 40136 Bologna, Italy,
	{\tt stefano.demiranda@unibo.it}},
C. Lovadina\thanks{
Dipartimento di Matematica, Universit\`a di Milano, Via Saldini 50, 20133 Milano,
and IMATI del CNR, Via Ferrata 1, 27100 Pavia, Italy,
 {\tt carlo.lovadina@unimi.it}},
L. Patruno\thanks{DICAM, University of Bologna, Viale Risorgimento 2, 40136 Bologna, Italy, {\tt luca.patruno@unibo.it}}
}

\date{}

\title{A Stress/Displacement Virtual Element Method for Plane Elasticity Problems}

\begin{document}

\maketitle

\begin{abstract}
The numerical approximation of 2D elasticity problems is considered, in the framework of the small strain theory and in connection with the mixed Hellinger-Reissner variational formulation. A low-order Virtual Element Method (VEM) with a-priori symmetric stresses is proposed. Several numerical tests are provided, along with a rigorous stability and convergence analysis.
\end{abstract}

{
\section{Introduction}

The Virtual Element Method (VEM) is a new technology for the approximation of partial differential equation problems. VEM was born in 2012, see \cite{volley}, as an evolution of modern mimetic schemes (see for instance \cite{Brezzi-Lipnikov-Shashkov-Simoncini:2007,MFD-book,Brezzi-Buffa-Lipnikov}), which shares the same variational background of the Finite Element Method (FEM). The initial motivation of VEM is the need to construct an accurate {\em conforming} Galerkin scheme with the capability to deal with highly general polygonal/polyhedral meshes, including ``hanging vertexes'' and non-convex shapes.
The virtual element method reaches this goal by {\em abandoning the local polynomial approximation concept}, and uses, instead, approximating functions which are solution to suitable local partial differential equations (of course, connected with the original problem to solve). Therefore, in general, the discrete functions are not known pointwise, but a limited information of them is at disposal. The key point is that the available information are indeed sufficient to implement the stiffness matrix and the right-hand side. We remark that VEM is not the only available technology for dealing with polytopal meshes: a brief representative sample of the increasing list of technologies that make use of polygonal/polyhedral meshes can be found in
\cite{%
	BLS05bis,BLM11,Brezzi-Buffa-Lipnikov,BLM11book,Bishop13,JA12,LMSXX,MS11,NBM09,RW12,POLY37,%
	ST04,TPPM10,VW13,Wachspress11,DiPietro-Ern-1,Gillette-1,Wang-1,PolyDG-1}. %
We here recall, in particular, the polygonal finite elements and the mimetic discretisation schemes.
However, VEM is experiencing a growing interest towards Structural Mechanics problems, also in the engineering community. We here cite the recent works \cite{GTP14,BeiraoLovaMora,ABLS_part_one,ABLS_part_two,wriggers,BCP,ANR} and \cite{BeiraodaVeiga-Brezzi-Marini:2013,Brezzi-Marini:2012}, for instance.

In the present paper we apply the VEM concept to two-dimensional elasticity problems in the framework of small displacements and small deformations. More precisely, we consider the (mixed) Hellinger-Reissner functional (see, for instance, \cite{BoffiBrezziFortin,Braess:book}) as the starting point of the discretization procedure. Thus, the numerical scheme approximates both the stress and the displacement fields. 

It is well-known that in the Finite Element practice, designing a stable and accurate element for the Hellinger-Reissner functional, is not at all a trivial task. Essentially, one is led either to consider quite cumbersome schemes, or to relax the symmetry of the Cauchy stress field, or to employ composite elements (a discussion about this issue can be found in \cite{BoffiBrezziFortin}, for instance).
We here exploit the flexibility of the VEM approach to propose and study {\em a low-order scheme, with a-priori symmetric Cauchy stresses, that can be used for general polygons, from triangular shapes on}. Furthermore, the method is robust with respect to the compressibility parameter, and therefore can be used for nearly incompressible situations. Our scheme approximates the stress field by using traction degrees of freedom (three per each edge), while the displacement field inside each polygon is essentially a rigid body motion. The VEM concept is then applied essentially for the stress field. We also remark that the construction of the discrete stress field is somehow similar to the construction of the discrete velocity field used for the Stokes problem in \cite{BLV}.
Instead, the displacement field is modelled with polynomial functions, in accordance with the classical Finite Element procedure.

An outline of the paper is as follows.
In Section \ref{sec:1} we briefly introduce the Hellinger-Reissner variational formulation of the elasticity problem. Section \ref{s:HR-VEM} concerns with the discrete problem: all the bilinear and linear forms are introduced and detailed.
Numerical experiments are reported in Section \ref{s:numer}, where suitable error measures are considered. These numerical tests are supported by the stability and convergence analysis developed in Section \ref{s:theoretical}. Finally, Section \ref{s:conclusions} draws some conclusions, including possible future extensions of the present study.

Throughout the paper, given two quantities $a$ and $b$, we use the notation $a\lesssim b$ to mean: there exists a constant $C$, independent of the mesh-size, such that $a\leq C\, b$. Moreover, we use standard notations for Sobolev spaces, norms and semi-norms (cf. \cite{Lions-Magenes}, for example).

\section{The elasticity problem in mixed form}\label{sec:1}

In this section we briefly present the elasticity problem as it stems from the Hellinger-Reissner principle. More details can be found in \cite{BoffiBrezziFortin, Braess:book}.


\begin{equation}\label{strong}
\left\lbrace{
	\begin{aligned}
	&\mbox{Find } (\bfsigma,\bbu)~\mbox{such that}\\
	&-\bdiv \bfsigma= \bbf\qquad \mbox{in $\Omega$}\\
	& \bfsigma = \C \teps(\bbu)\qquad \mbox{in $\Omega$}\\
	&\bbu_{|\partial\Omega}=\bfzero
	\end{aligned}
} \right.
\end{equation}

Defining $(\cdot,\cdot)$ as the scalar product in $L^2$,
and $a(\bfsigma,\bftau):=(\D \bfsigma, \bftau)$, a mixed variational formulation of the problem reads:
\begin{equation}\label{cont-pbl}
\left\lbrace{
\begin{aligned}
&\mbox{Find } (\bfsigma,\bbu)\in \Sigma\times U~\mbox{such that}\\
&a(\bfsigma,\bftau) + (\bdiv \bftau, \bbu)=0 \quad \forall \bftau\in \Sigma\\
& (\bdiv \bfsigma, \bbv) = -(\bbf,\bbv)\quad \forall \bbv\in U
\end{aligned}
} \right.
\end{equation}
where $\O\subset \R^2$ is a polygonal domain, $\Sigma=H(\bdiv;\Omega)$, $U=\times L^2(\Omega)^2$, and the loading $\bbf \in L^2(\O)^2$. We recall that $\bdiv$ is the vector-valued divergence operator, acting on a second order tensor field. Thus, $\bdiv \bftau$ is, in Cartesian components: $\frac{\partial\tau_{ij}}{\partial x_j}$ (Einstein's  summation convention is here adopted).
The elasticity fourth-order symmetric tensor $\D:=\C^{-1}$ is assumed to be uniformly bounded and positive-definite.
It is well known that problem \eqref{cont-pbl} is well-posed (see \cite{BoffiBrezziFortin}, for instance). in particular, it holds:

\begin{equation}\label{eq:aprioriest}
||\bfsigma||_\Sigma + || \bbu ||_U \leq C || \bbf||_0 ,
\end{equation}
where $C$ is a constant depending on $\Omega$ and on the material tensor $\C$.

Note also that the bilinear form $a(\cdot,\cdot)$ in~\eqref{cont-pbl} can obviously be split as
\begin{equation}\label{dec_a} a(\bfsigma,\bftau)=\sum_{{\E\in \Th}}a_E(\bfsigma,\bftau) \quad \textrm{ with } \quad
a_E(\bfsigma,\bftau) : = \int_E \D \bfsigma : \bftau
\end{equation}
for all $\bfsigma,\bftau \in \Sigma$. Above, $\Th$ is a polygonal mesh of meshsize $h$.

Similarly, it holds
\begin{equation}\label{dec_div}
(\bdiv \bftau, \bbv)=\sum_{{\E\in \Th}}(\bdiv \bftau, \bbv)_E \quad \textrm{ with } \quad
(\bdiv \bftau,\bbv)_E : = \int_E \ \bdiv \bftau \cdot \bbv ,
\end{equation}
for all $(\bftau,\bbv) \in \Sigma\times U$.

\begin{remark}\label{rm:incompress}
As discussed in \cite{BoffiBrezziFortin}, estimate \eqref{eq:aprioriest} does not break down for nearly incompressible materials. More precisely, considering the constitutive law:

\begin{equation}\label{eq:hom-iso}
\C \bfvarepsilon = 2\mu \bfvarepsilon + \lambda {\rm tr}(\bfvarepsilon) Id \qquad \forall \mbox{ {\rm symmetric tensor} \bfvarepsilon},
\end{equation}
with $\lambda, \mu > 0$ the Lame's parameters and ${\rm tr}(\cdot)$ the trace operator, the constant $C$ in \eqref{eq:aprioriest} can be chosen independent of $\lambda$. The key point is that it is sufficient to check the $\Sigma$-coercivity of the bilinear form $a(\cdot,\cdot)$ in \eqref{cont-pbl} for the subspace:

\begin{equation}\label{eq:kernel}
K = \left\{ \bftau\in \Sigma \, :\, (\bdiv \bftau ,\bbv)=0 \quad \forall \bbv\in U \right\}.
\end{equation}

In fact, there exists a positive constant $\alpha$ such that (see \cite{BoffiBrezziFortin}):

\begin{equation}
a(\bftau,\bftau)\ge \alpha ||\bftau||^2_\Sigma \qquad \forall \bftau\in K ,
\end{equation}
with $\alpha$ independent of $\lambda$.
	
\end{remark}
%

\section{The Virtual Element Method}
\label{s:HR-VEM}

We outline the Virtual Element discretization of problem \eqref{cont-pbl}.
Let $\{\mathcal{T}_h\}_h$ be a sequence of decompositions of $\Omega$ into general polygonal elements $E$ with
\[
h_E := {\rm diameter}(E) , \quad
h := \sup_{E \in \mathcal{T}_h} h_E .
\]
In what follows, $|E|$ and $|e|=h_e$ will denote the area of $E$ and the length of the side $e\in\partial E$, respectively.

We suppose that for all $h$, each element $E$ in $\mathcal{T}_h$ fulfils the following assumptions:
\begin{itemize}
	\item $\mathbf{(A1)}$ $E$ is star-shaped with respect to a ball of radius $ \ge\, \gamma \, h_E$,
	\item $\mathbf{(A2)}$ the distance between any two vertexes of $E$ is $\ge c \, h_E$,
\end{itemize}
where $\gamma$ and $c$ are positive constants. We remark that the hypotheses above, though not too restrictive in many practical cases,
can be further relaxed, as noted in ~\cite{volley}.

In addition, we suppose that the tensor $\D$ is piecewise constant with respect to the underlying mesh $\Th$.

\subsection{The local spaces}\label{ss:E-spaces}

Given a polygon $E\in\Th$ with $n_E$ edges, we first introduce the space of local infinitesimal rigid body motions:

\begin{equation}\label{eq:rigid}
RM(E)=\left\{ \bbr(\bbx) = \bba + b(\bbx -\bbx_C)^\perp \quad \bba\in\R^2, \ b\in \R  \right\}.
\end{equation}
Here above, given $\bbc=(c_1,c_2)^T\in\R^2$, $\bbc^\perp$ is the counterclock-wise rotated vector  $\bbc^\perp=(c_2,-c_1)^T$, and $\bbx_C$ is the baricenter of $E$. For each edge $e$ of $\partial E$, we introduce the space

\begin{equation}\label{eq:edge_approx}
R(e)=\left\{ \bbt(s) = \bbc + d\,s\,\bbn \quad \bbc\in\R^2, \ d\in \R, \ s\in [-1/2,1/2]  \right\}.
\end{equation}
Here above, $s$ is a local linear coordinate on $e$, such that $s=0$ corresponds to the edge midpoint. Furthermore, $\bbn$ is the outward normal to the edge $e$. Hence, $R(e)$ consists of vectorial functions which have the edge tangential component constant, and the edge normal component linear along the edge.
Our local approximation space for the  stress field is then defined by

\begin{equation}\label{eq:local_stress}
\begin{aligned}
 \Sigma_h(E)=\Big\{ \bftau_h\in & H(\bdiv;E)\ :\ \exists \bbw^\ast\in H^1(E)^2 \mbox{ such that } \bftau_h=\C\teps(\bbw^\ast);\\  &(\bftau_h\,\bbn)_{|e}\in R(e) \quad \forall e\in \partial E;\quad
\bdiv\bftau_h\in RM(E) \Big\}.
\end{aligned}
\end{equation}

\begin{remark} Alternatively, the space \eqref{eq:local_stress} can be defined as follows.
	\begin{equation}\label{eq:local_stress-alt}
	\begin{aligned}
	 \Sigma_h(E)=\Big\{ \bftau_h\in & H(\bdiv;E)\ :\ \bftau_h=\bftau_h^T;\quad \curl \bcurl(\D \bftau_h) = 0 ; \\  &(\bftau_h\,\bbn)_{|e}\in R(e) \quad \forall e\in \partial E;\quad
	 \bdiv\bftau_h\in RM(E) \Big\}.
	\end{aligned}
	\end{equation}
Here above, the equation $\curl \bcurl (\D \bftau_h) = 0$ is to be intended in the distribution sense.	
\end{remark}

We remark that, once $(\bftau_h\,\bbn)_{|e}=\bbc_e +d_es\,\bbn$ is given for all $e\in\partial E$, cf. \eqref{eq:edge_approx}, the quantity $\bdiv\bftau_h\in RM(E)$ is determined. Indeed, denoting with $\bfvarphi:\partial E\to \R^2$ the function such that $\bfvarphi_{|e}:=\bbc_e+d_es\,\bbn$, the obvious compatibility condition

\begin{equation}\label{eq:compat}
\int_E \bdiv\bftau_h\cdot \bbr =
\int_{\partial E}\bfvarphi\cdot \bbr \qquad \forall \bbr\in RM(E) ,
\end{equation}
allows to compute $\bdiv\bftau_h$ using the $\bbc_e$'s and the $d_e$'s. More precisely, setting (cf \eqref{eq:rigid})

\begin{equation}\label{eq:div1}
\bdiv\bftau_h = \bfalpha_E + \beta_E (\bbx -\bbx_C)^\perp ,
\end{equation}
from \eqref{eq:compat} we infer

\begin{equation}\label{eq:div2}
\left\lbrace{
	\begin{aligned}
&\bfalpha_E =\frac{1}{|E|}\int_{\partial E}\bfvarphi = \frac{1}{|E|}\sum_{e\in\partial E}\int_e \bbc_e \\
&\beta_E = \frac{1}{ \int_E | \bbx -\bbx_C |^2 }\int_{\partial E} \bfvarphi\cdot (\bbx -\bbx_C)^\perp  = \frac{1}{ \int_E | \bbx -\bbx_C |^2 }\sum_{e\in\partial E}\int_e (\bbc_e+d_e s\,\bbn)\cdot (\bbx -\bbx_C)^\perp.
\end{aligned}
} \right.
\end{equation}

The local approximation space for the displacement field is simply defined by, see \eqref{eq:rigid}:

\begin{equation}\label{eq:local_displ}
U_h(E)=\Big\{ \bbv_h\in  L^2(E)^2\ :\ \bbv_h\in RM(E) \Big\}.
\end{equation}

We notice that $\dim( \Sigma_h(E))=3\,n_E$, while $\dim(U_h(E))=3$.

\subsection{The local bilinear forms}\label{ss:E-bforms}

Given $E\in \Th$, we first notice that, for every $\bftau_h\in  \Sigma_h(E)$ and $\bbv_h\in U_h(E)$, the term

\begin{equation}\label{eq:div-ex}
\int_E \bdiv \bftau_h\cdot \bbv_h
\end{equation}
is computable from the knowledge of the degrees of freedom. Therefore, there is no need to introduce any approximation in the structure of the terms $(\bdiv \bftau, \bbu)$ and $(\bdiv \bfsigma, \bbv)$ in problem \eqref{cont-pbl}. Instead, the term

\begin{equation}
a_E(\bfsigma_h,\bftau_h)  = \int_E \D \bfsigma_h : \bftau_h
\end{equation}
is not computable for a general couple $(\bfsigma_h,\bftau_h)\in \Sigma_h(E)\times \Sigma_h(E)$. As usual in the VEM approach (see \cite{volley}, for instance), we then need to introduce a suitable approximation $a_E^h(\cdot,\cdot)$ of $a_E(\cdot,\cdot)$. To this end, we first define the projection operator

\begin{equation}\label{eq:proj}
\left\{
\begin{aligned}
& \Pi_E \, :\,  \Sigma_h(E)\to \P_0(E)^{2\times 2}_s\\
& \bftau_h \mapsto \Pi_E\bftau_h\\
&a_E(\Pi_E\bftau_h,\bfpi_0)  = a_E(\bftau_h,\bfpi_0)\qquad \forall\bfpi_0 \in \P_0(E)^{2\times 2}_s
\end{aligned}
\right.
\end{equation}

Above and in the sequel, given a domain $\omega$ and an integer $k\ge 0$, the space $\P_k(\omega)$ denotes the polynomials up to degree $k$, defined on $\omega$. Furthermore, given a functional space $X$, $X^{2\times 2}_s$ denotes the $2\times 2$ symmetric tensors whose components belong to $X$. Therefore, the operator in \eqref{eq:proj} is a projection onto the piecewise constant symmetric tensors.

We then set

\begin{equation}\label{eq:ah1}
\begin{aligned}
a_E^h(\bfsigma_h,\bftau_h)  &=
a_E(\Pi_E\bfsigma_h,\Pi_E\bftau_h) + s_E\left( (Id-\Pi_E)\bfsigma_h, (Id-\Pi_E)\bftau_h \right)\\
&=\int_E \D (\Pi_E\bfsigma_h) : (\Pi_E\bftau_h)  + s_E\left( (Id-\Pi_E)\bfsigma_h, (Id-\Pi_E)\bftau_h \right) ,
\end{aligned}
\end{equation}
where $s_E(\cdot,\cdot)$ is a suitable stabilization term. We propose the following choice:
	
	\begin{equation}\label{eq:stab1}
	s_E(\bfsigma_h,\bftau_h) : = \kappa_E\, h_E\int_{\partial E} \bfsigma_h\bbn\cdot \bftau_h\bbn  ,
	\end{equation}
	where $\kappa_E$ is a positive constant to be chosen. For instance, in the numerical examples of Section \ref{s:numer}, $\kappa_E$ is set equal to $\frac{1}{2} {\rm tr}(\D_{|E})$; however, any norm of $\D_{|E}$ can be used. A possible variant of \eqref{eq:stab1} is provided by
	
	\begin{equation}\label{eq:stab1bis}
	s_E(\bfsigma_h,\bftau_h) : = \kappa_E\, \sum_{e\in\partial E}h_e\int_{e} \bfsigma_h\bbn\cdot \bftau_h\bbn.
	\end{equation}


\subsection{The local loading terms}\label{ss:E-load}

We need to consider the term, see \eqref{cont-pbl}:

\begin{equation}\label{eq:fh}
(\bbf,\bbv_h)=\int_{\Omega}\bbf\cdot\bbv_h =\sum_{E\in\Th}\int_{E}\bbf\cdot\bbv_h .
\end{equation}
We remark that, since $\bbv_h\in RM(E)$, computing \eqref{eq:fh} is possible once a suitable quadrature rule is available for polygonal domains. For such an issue, see for instance \cite{NME2759,Sommariva2007,MS11}.


\subsection{The discrete scheme}\label{ss:discrete}

We are now ready to introduce the discrete scheme. We introduce a global approximation space for the stress field, by glueing the local approximation spaces, see \eqref{eq:local_stress}:

\begin{equation}\label{eq:global-stress}
\Sigma_h=\Big\{ \bftau_h\in  H(\bdiv;\Omega)\ :\ \bftau_{h|E}\in  \Sigma_h(E)\quad \forall E\in\Th \Big\}.
\end{equation}

For the global approximation of the displacement field, we take, see \eqref{eq:local_displ}:

\begin{equation}\label{eq:global-displ}
U_h=\Big\{ \bbv_h\in  L^2(\Omega)^2\ :\ \bbv_{h|E}\in U_h(E)\quad \forall E\in\Th \Big\}.
\end{equation}

Furthermore, given a local approximation of $a_E(\cdot,\cdot)$, see \eqref{eq:ah1}, we set

\begin{equation}\label{eq:global-ah}
a_h(\bfsigma_h,\bftau_h):= \sum_{E\in\Th}a_E^h(\bfsigma_h,\bftau_h) .
\end{equation}

The method we consider is then defined by

\begin{equation}\label{eq:discr-pbl-ls}
\left\lbrace{
	\begin{aligned}
	&\mbox{Find } (\bfsigma_h,\bbu_h)\in \Sigma_h\times U_h~\mbox{such that}\\
	&a_h(\bfsigma_h,\bftau_h)    + (\bdiv \bftau_h, \bbu_h)= 0 \quad \forall \bftau_h\in \Sigma_h\\
	& (\bdiv \bfsigma_h, \bbv_h) = -(\bbf,\bbv_h)\quad \forall \bbv_h\in U_h .
	\end{aligned}
} \right.
\end{equation}

Introducing the bilinear form $\A_h:(\Sigma_h\times U_H)\times (\Sigma_h\times U_h)\to \R$ defined by

\begin{equation}\label{eq:totbilinh}
\begin{aligned}
\A_h(\bfsigma_h,\bbu_h ;\bftau_h,\bbv_h):=
a_h(\bfsigma_h,\bftau_h)  + (\bdiv \bftau_h, \bbu_h)+(\bdiv \bfsigma_h, \bbv_h) ,
\end{aligned}
\end{equation}
problem \eqref{eq:discr-pbl-ls} can be written as

\begin{equation}\label{discr-pbl-cpt}
\left\lbrace{
	\begin{aligned}
	&\mbox{Find } (\bfsigma_h,\bbu_h)\in \Sigma_h\times U_h~\mbox{such that}\\
	&\A_h(\bfsigma_h,\bbu_h;\bftau_h,\bbv_h)= - (\bbf,\bbv_h)\quad \forall (\bftau_h,\bbv_h)\in \Sigma_h\times U_h .
	\end{aligned}
} \right.
\end{equation}

We will prove in Section \ref{s:theoretical} that our method is first order convergent with respect to the natural norms, see in particular Theorem \ref{th:main_convergence}. More precisely, the following error estimate holds true.

\begin{equation}\label{eq:main_conv-est-preview}
|| \bfsigma - \bfsigma_h||_{\Sigma} + || \bbu - \bbu_h||_U \lesssim C\, h ,
\end{equation}
where $C=C(\Omega,\bfsigma,\bbu)$ is independent of $h$ but depends on the domain $\Omega$ and on the Sobolev regularity of $\bfsigma$ and $\bbu$.


\section{Numerical results}\label{s:numer}
The present section is devoted to the validation of the proposed methodology through the assessment of accuracy on a selected number of test problems. Applicability to structural analysis is then demonstrated through a classical benchmark.

\subsection{Accuracy assessment}
\label{ss:analytical}
We consider two boundary value problems on the unit square domain $\Omega = [0, 1]^2$, with known analytical solution, discussed in \cite{BeiraoLovaMora,ABLS_part_one}. The material obeys to a homogeneous isotropic constitutive law, see \eqref{eq:hom-iso}, with material parameters assigned in terms of the Lam\'e constants, here set as $\lambda = 1$ and $\mu = 1$. Plane strain regime is invoked throughout. The tests are defined by choosing a required solution and deriving the corresponding body load $\bbf$, as synthetically indicated in the following:
\begin{itemize}
	\item[$\bullet$]
	Test $a$
	\begin{eqnarray}
		\label{eq:test_2a_sol}
		\left\{
		\begin{array}{l}
			u_1 = x^3 - 3 x y^2 \\
			u_2 = y^3 - 3 x^2 y \\
			\bbf = \bfzero
		\end{array}
		\right .
	\end{eqnarray}
	\item[$\bullet$]
	Test $b$
	\begin{eqnarray}
		\label{eq:test_2b_sol}
		\left\{
		\begin{array}{l}
			u_1 = u_2 = \sin(\pi x)  \sin(\pi y) \\
			f_1 = f_2 = -\pi^2 \left[ -(3 \mu + \lambda ) \sin(\pi x) \sin ( \pi y) + ( \mu + \lambda ) \cos ( \pi x) \cos ( \pi y ) \right]
		\end{array}
		\right .
	\end{eqnarray}
\end{itemize}
As it can be observed, Test $a$ is a problem with Dirichlet non-homogeneous boundary conditions, zero loading and a polynomial solution; whereas Test $b$ has homogeneous Dirichlet boundary conditions, trigonometric distributed loads with a trigonometric solution.

In order to test the robustness of the proposed procedure with respect to element topology and mesh distortion, six different meshes are considered, as can be inspected in Fig. \ref{fig:meshes}. Three are {\em structured} meshes composed of triangles, quadrilaterals and a set of quads, pentagons and hexagons. In the following, such meshes are denoted by the letter "S". Three {\em unstructured} meshes are considered as well, comprising triangles, quadrilaterals and random polygons; these are denoted by the letter "U". In the numerical campaign the mesh size parameter is chosen to be the average edge length, denoted with $\bar{h}_e$. We remark that, under mesh assumptions $\mathbf{(A1)}$ and $\mathbf{(A2)}$ and for a quasi-uniform family of mesh, $\bar{h}_e$ is indeed equivalent to both $h_E$ and $h$.
The accuracy and the convergence rate assessment is carried out using the following error norms:

\begin{itemize}
	\item[$\bullet$] Discrete error norms for the stress field:
	\begin{equation}
		\label{eq:stress_err_norm}
		E_{\bfsigma}   :=\left( \sum_{e \in \Eh} \kappa_e\, |e|\int_{e} | (\bfsigma - \bfsigma_h)\bbn  |^2\right)^{1/2} ,
	\end{equation}
	where $\kappa_e= \kappa=\frac{1}{2} {\rm tr}(\D)$ (the material is here homogeneous).
	We remark that the quantity above scales like the internal elastic energy, with respect to the size of the domain and of the elastic coefficients.  	
	
	We make also use of the $L^2$ error on the divergence:
	
	\begin{equation}
		\label{eq:stress_div_err_norm}
		E_{\bfsigma, \bdiv}   :=\left( \sum_{E \in \Th} \int_E | \bdiv(\bfsigma - \bfsigma_h)  |^2\right)^{1/2} .
	\end{equation}

	\item[$\bullet$] $L^2$ error norm for the displacement field:
	\begin{equation}
		\label{eq:displ_err_norm}
		E_{\bbu} :=\left( \sum_{E \in \Th} \int_{E}  | \bbu - \bbu_h |^2 \right)^{1/2} = ||\bbu - \bbu_h ||_0.
	\end{equation}
	
\end{itemize}

\begin{figure}[h!]
	\centering
	\renewcommand{\thesubfigure}{}
	\subfigure[Tri (S)]{\includegraphics[width=0.3\textwidth,trim = 20mm 3mm 20mm 3mm, clip]{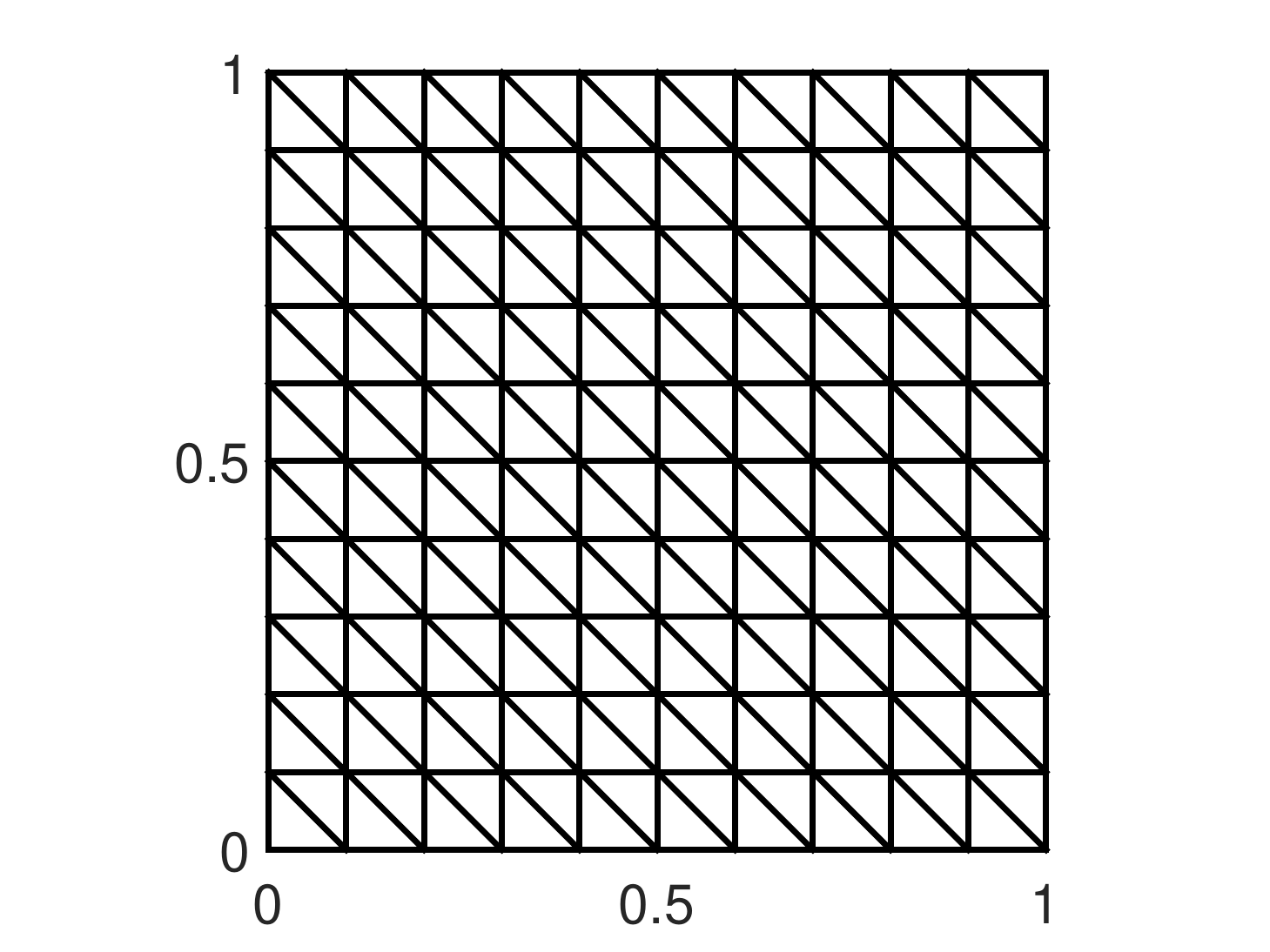}}
	\subfigure[Quad (S)]{\includegraphics[width=0.3\textwidth,trim = 20mm 3mm 20mm 3mm, clip]{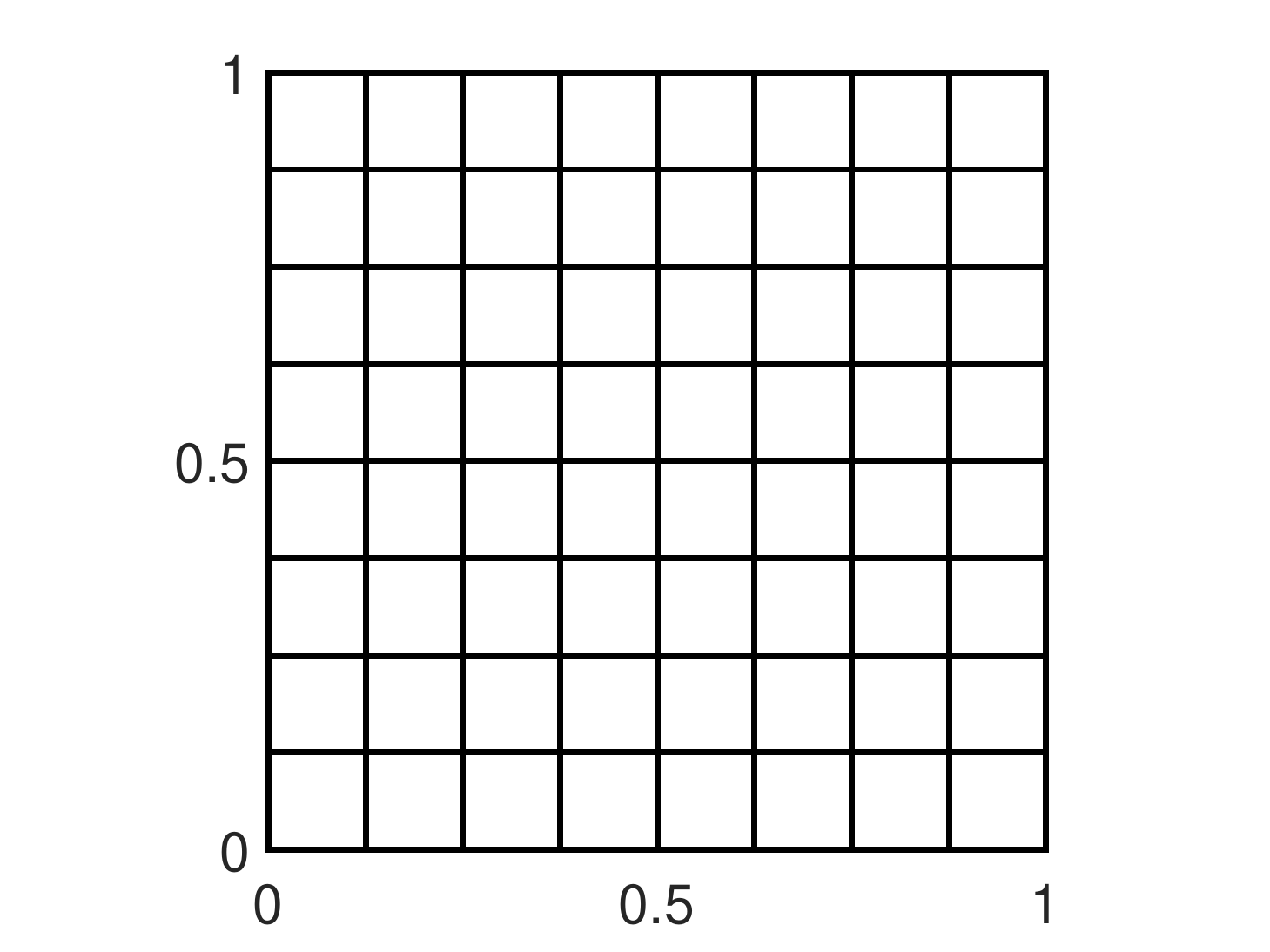}}
	\subfigure[Hex (S)]{\includegraphics[width=0.3\textwidth,trim = 20mm 3mm 20mm 3mm, clip]{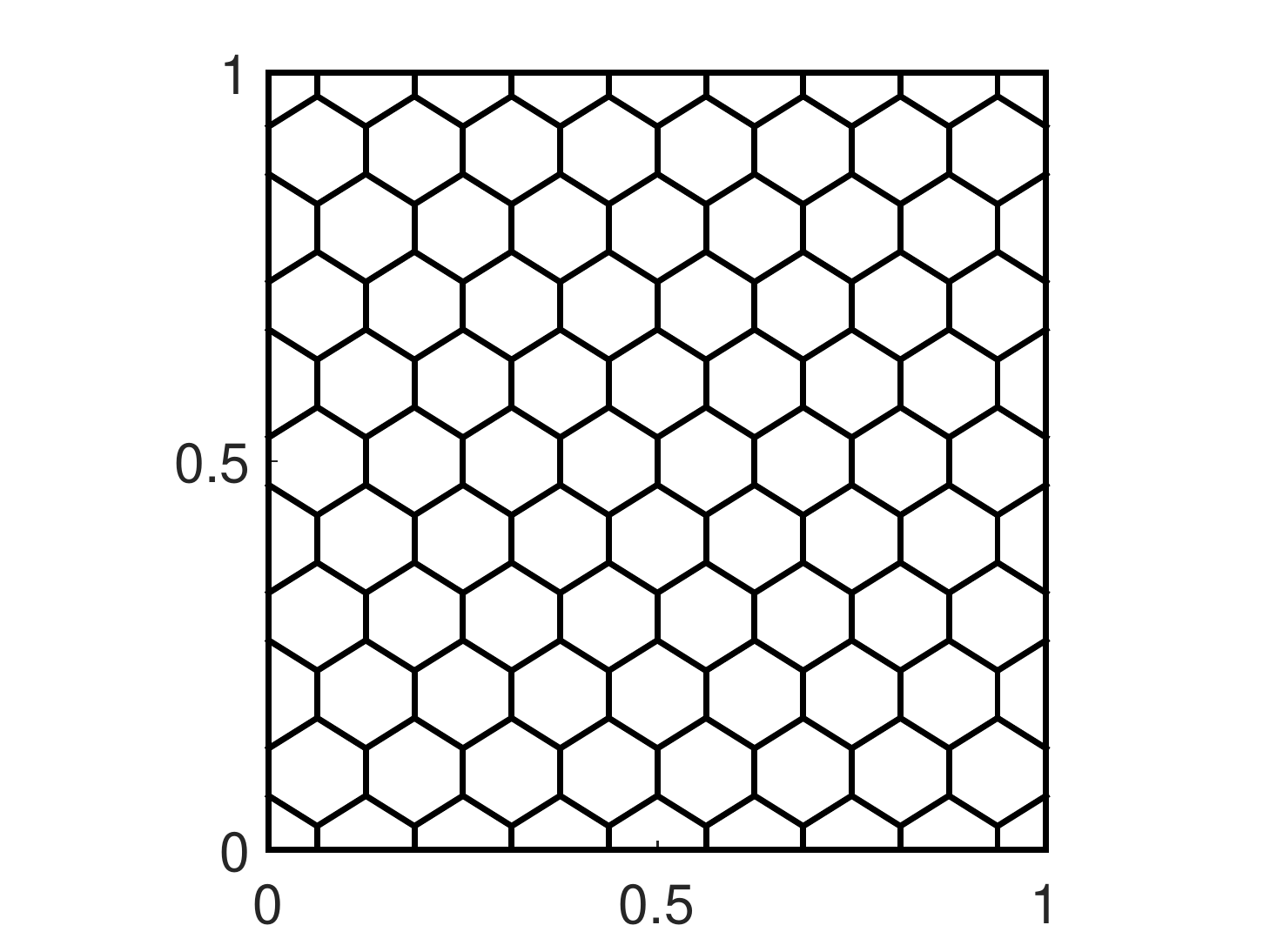}}
	\subfigure[Tri (U)]{\includegraphics[width=0.3\textwidth,trim = 20mm 3mm 20mm 3mm, clip]{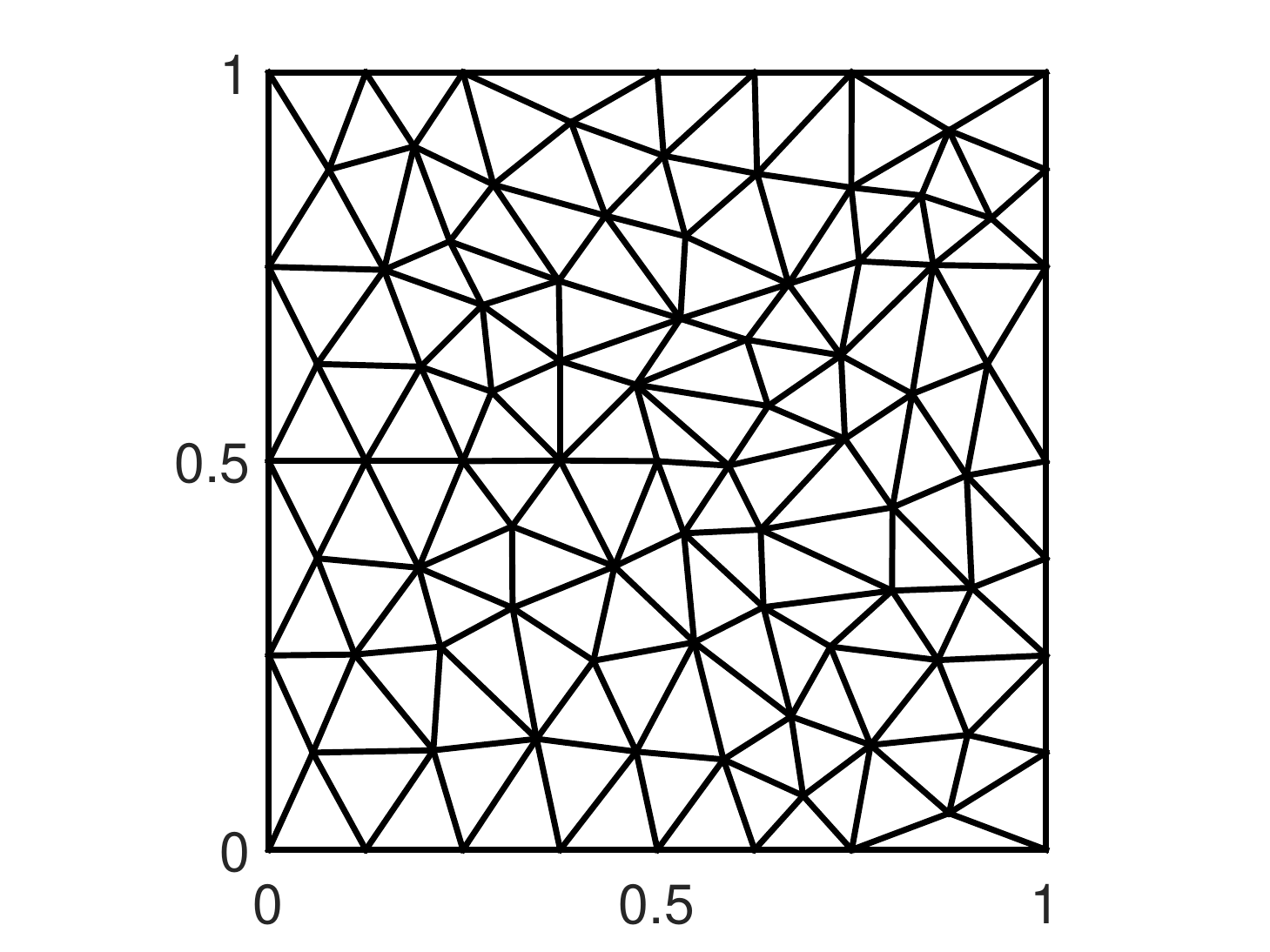}}
	\subfigure[Quad (U)]{\includegraphics[width=0.3\textwidth,trim = 20mm 3mm 20mm 3mm, clip]{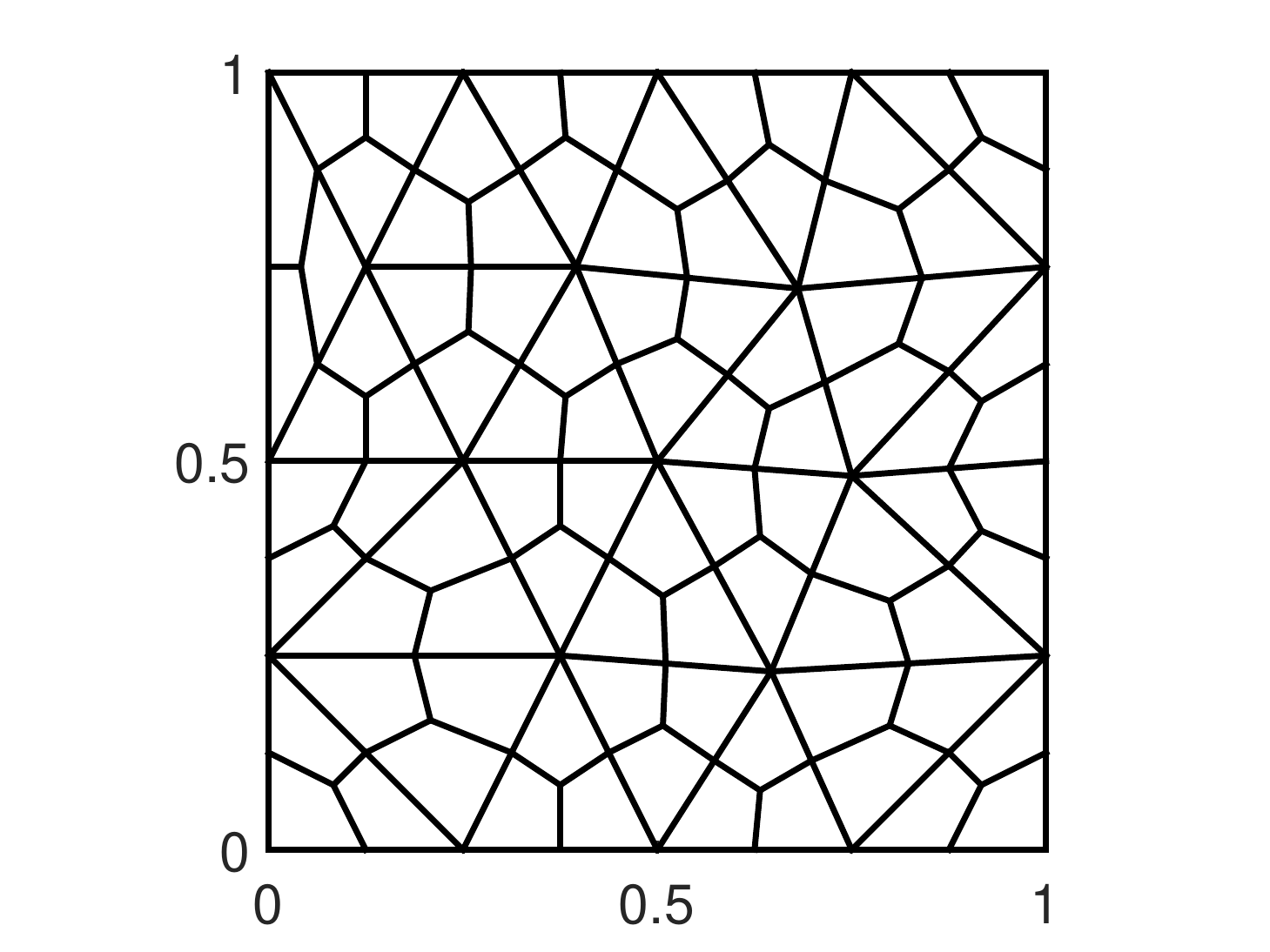}}
	\subfigure[Poly (U)]{\includegraphics[width=0.3\textwidth,trim = 20mm 3mm 20mm 3mm, clip]{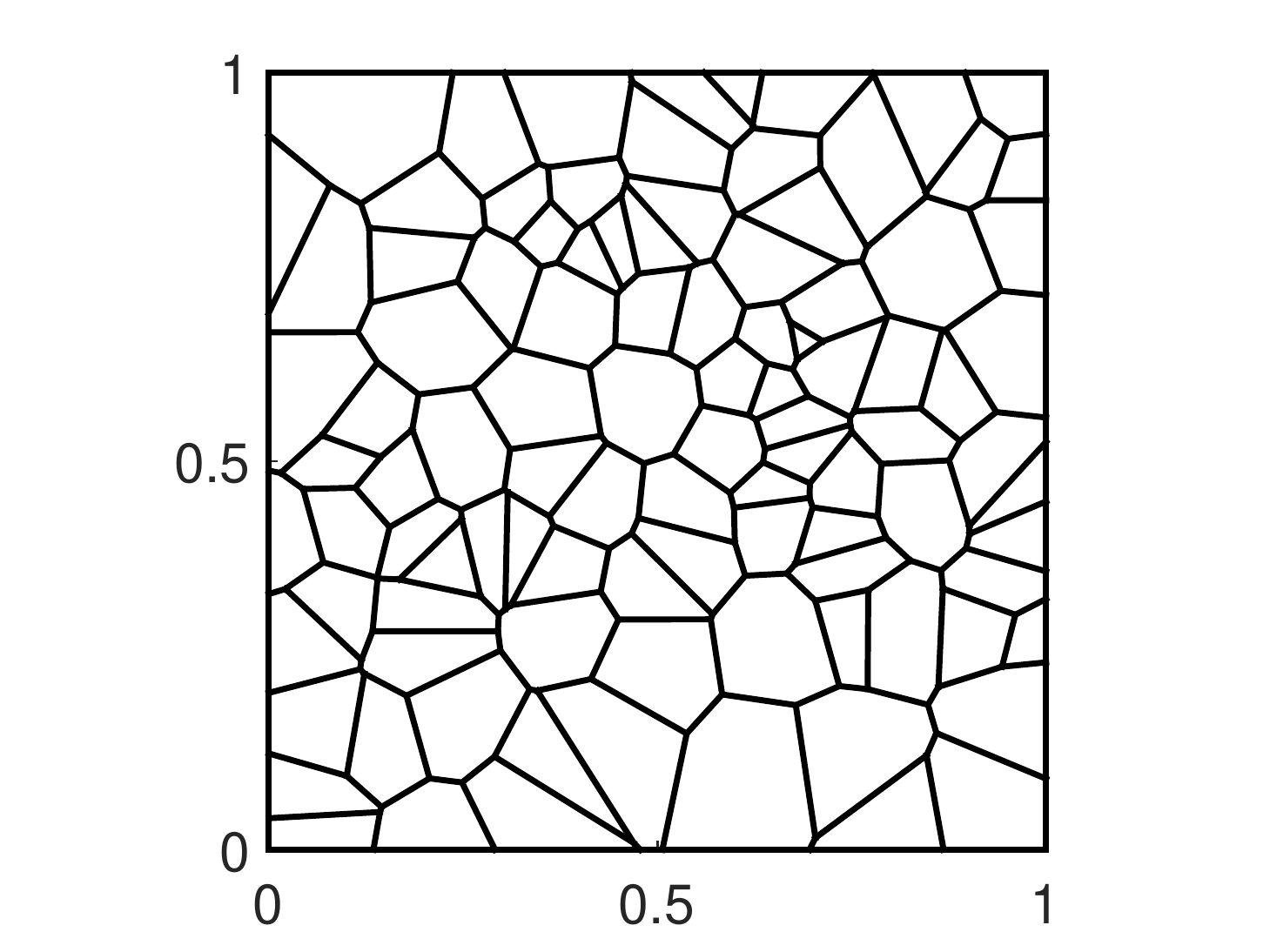}}
	\caption{Overview of adopted meshes for convergence assessment numerical tests.}
	\label{fig:meshes}
\end{figure}

Figure \ref{fig:resuTestA} reports the $\bar{h}_e-$convergence of the proposed method for Test $a$. As expected, the asymptotic convergence rate is approximately equal to $1$ for all the considered error norms and meshes. It is noted that, in this case, the $E_{\bfsigma, \bdiv}$ plots are not reported because such a quantity is captured up to machine precision for all the considered computational grids.

\begin{figure}[h!]
	\centering
	\subfigure[]{\includegraphics[width=0.45\textwidth,trim = 0mm 0mm 0mm 0mm, clip]{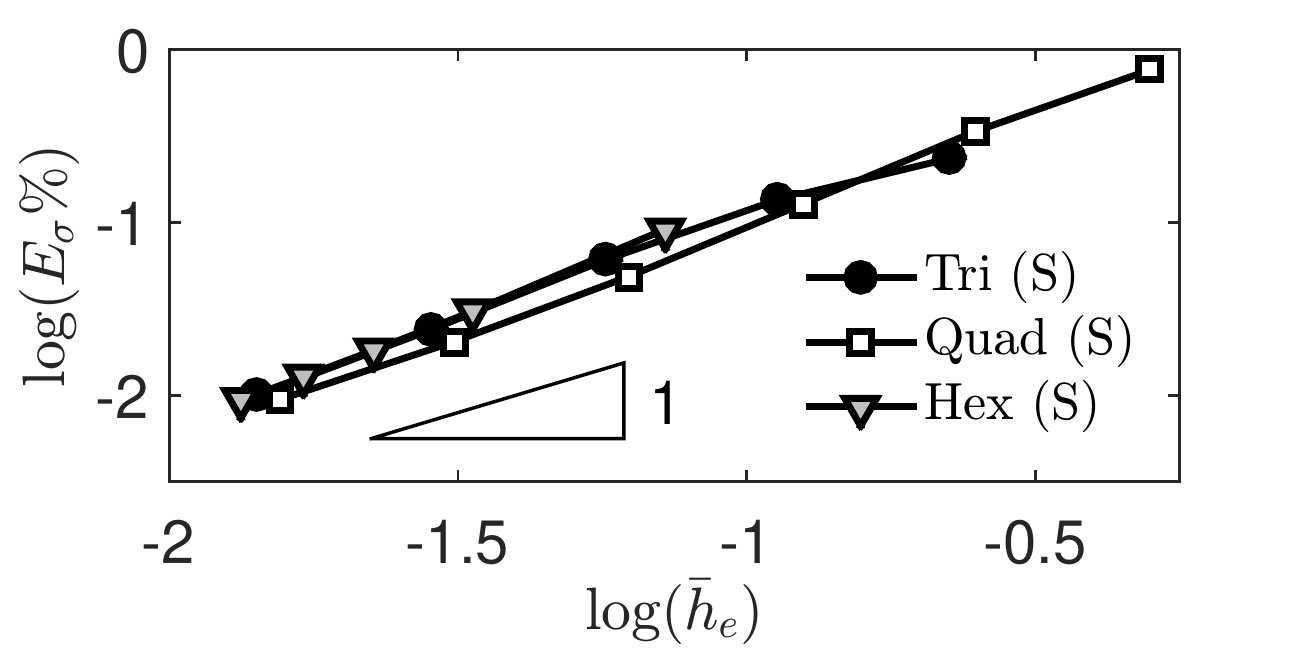}}
	\subfigure[]{\includegraphics[width=0.45\textwidth,trim = 0mm 0mm 0mm 0mm, clip]{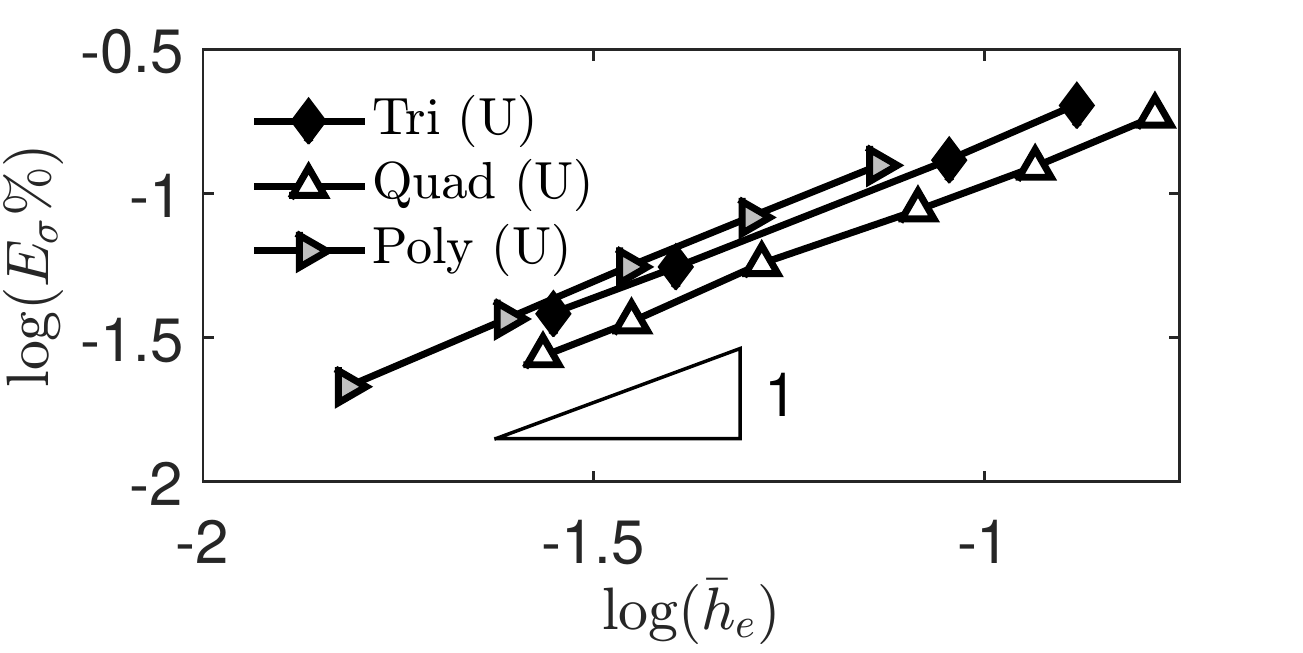}}\\
	\subfigure[]{\includegraphics[width=0.45\textwidth,trim = 0mm 0mm 0mm 0mm, clip]{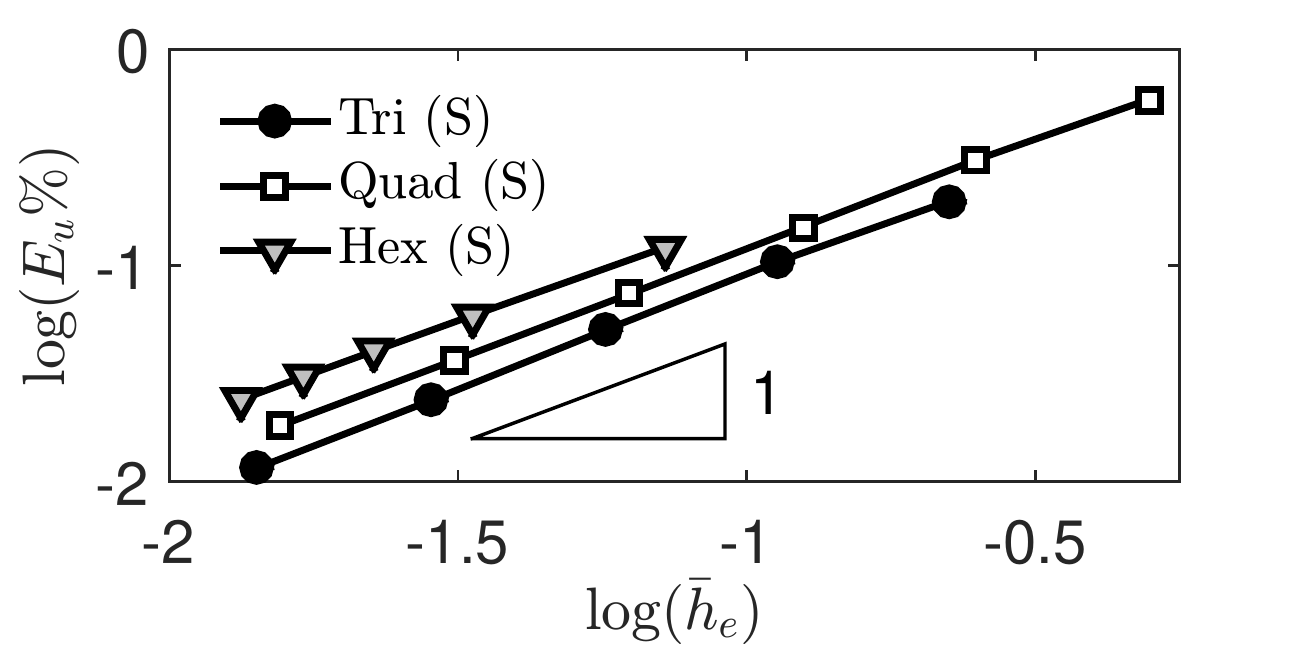}}
	\subfigure[]{\includegraphics[width=0.45\textwidth,trim = 0mm 0mm 0mm 0mm, clip]{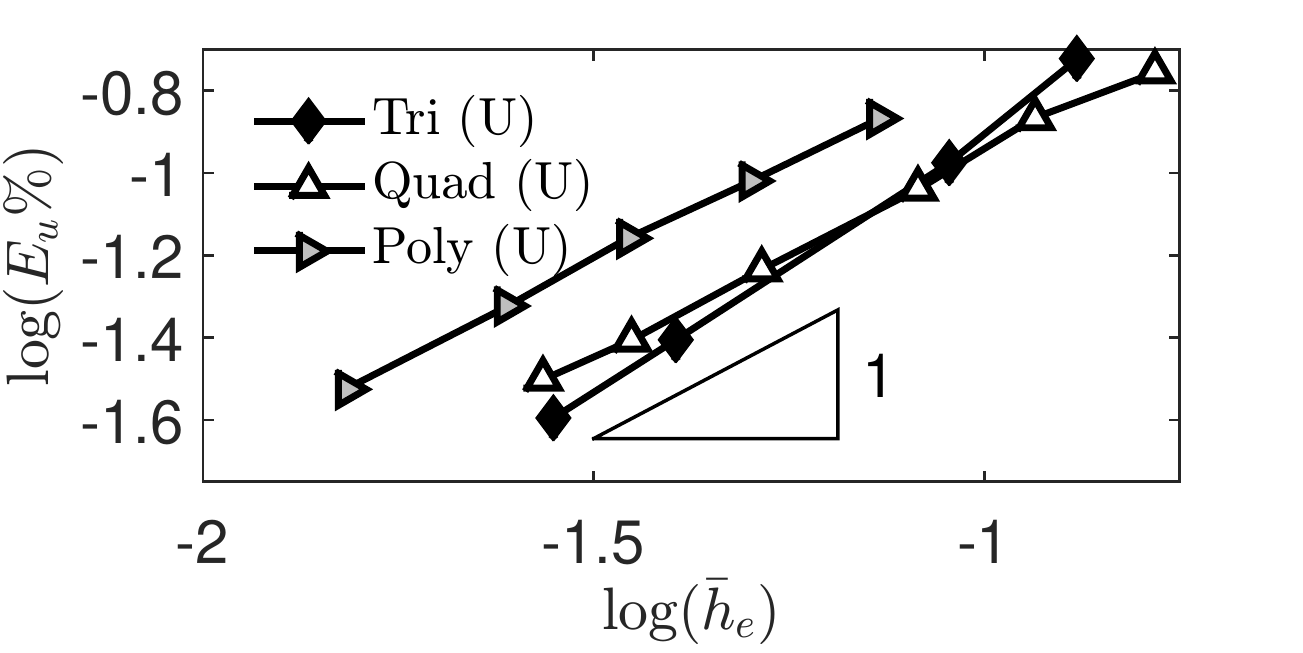}}
	\caption{$\bar{h}_e-$convergence results for Test $a$ on structured and unstructured meshes: (a) and (b) $E_{\bfsigma}$ error norm plots, (c) and (d) $E_{\bbu}$ error norm plots.}
	\label{fig:resuTestA}
\end{figure}

Figure \ref{fig:resuTestB} reports $\bar{h}_e-$convergence for Test $b$. Asymptotic converge rate is approximately equal to $1$ for all investigated mesh types and error measures, including $E_{\bfsigma, \bdiv}$. These results highlight the expected optimal performance of the proposed VEM approach and its robustness with respect to the adopted computational grid.

\begin{figure}[h!]
	\centering
	\subfigure[]{\includegraphics[width=0.45\textwidth,trim = 0mm 0mm 0mm 0mm, clip]{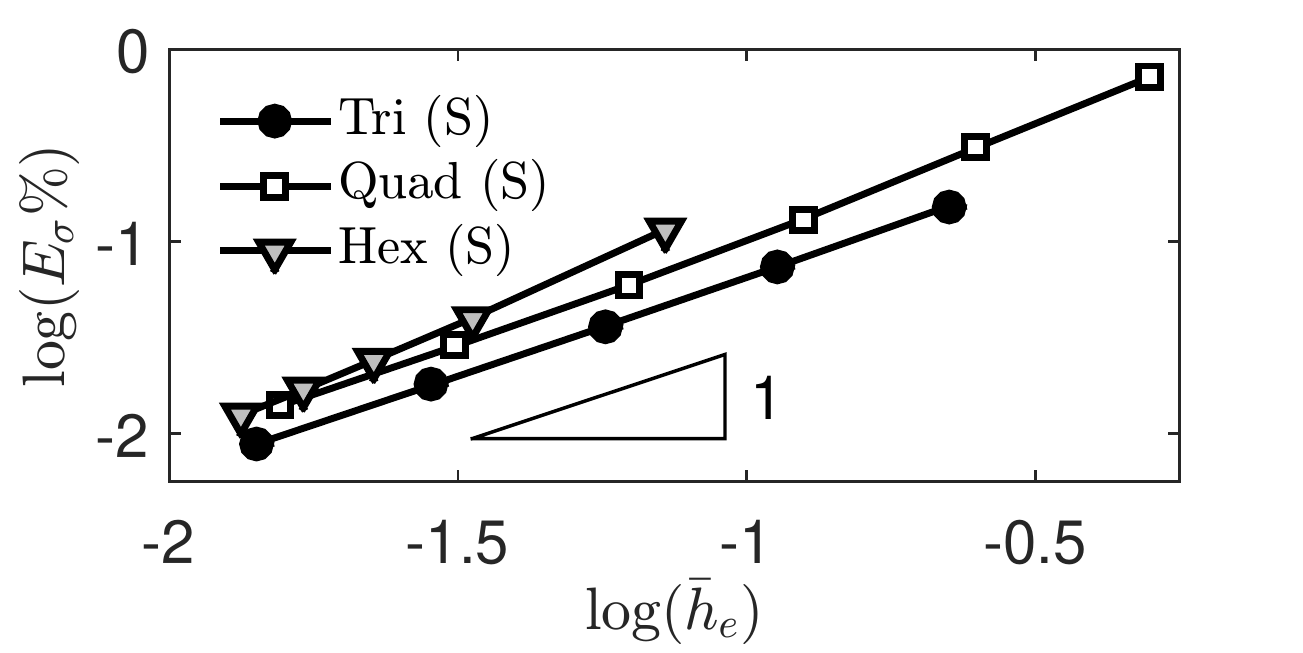}}
	\subfigure[]{\includegraphics[width=0.45\textwidth,trim = 0mm 0mm 0mm 0mm, clip]{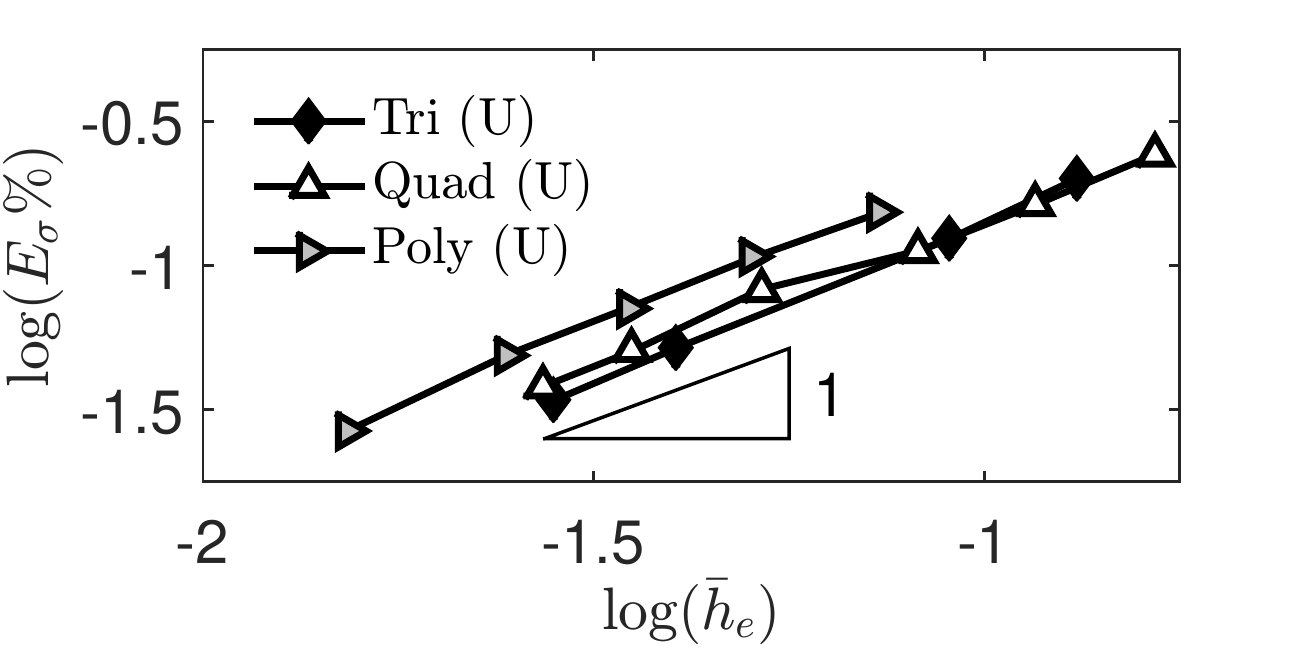}}\\
	\subfigure[]{\includegraphics[width=0.45\textwidth,trim = 0mm 0mm 0mm 0mm, clip]{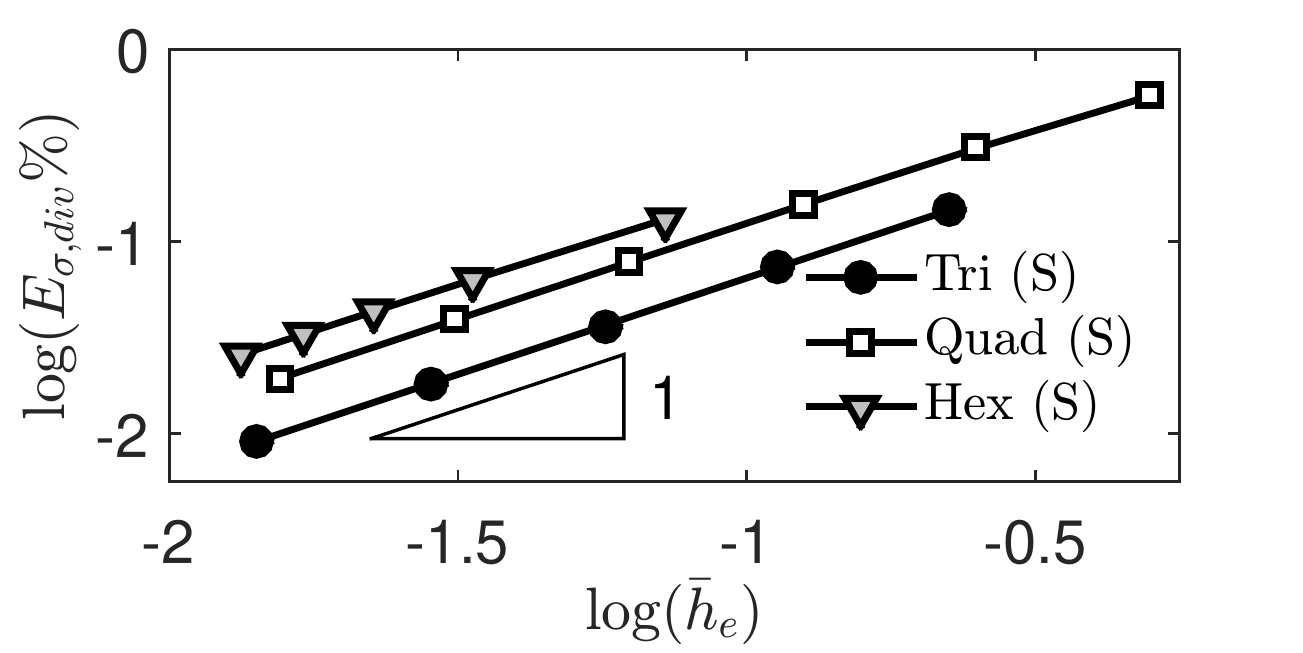}}
	\subfigure[]{\includegraphics[width=0.45\textwidth,trim = 0mm 0mm 0mm 0mm, clip]{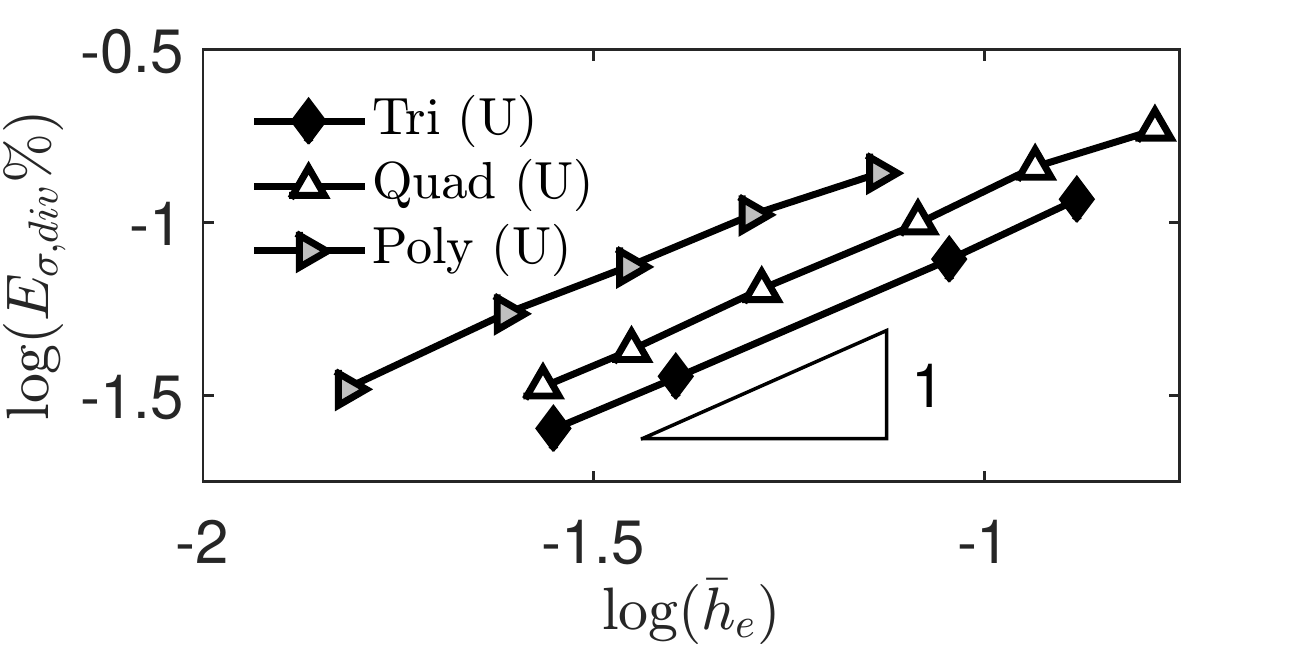}}\\
	\subfigure[]{\includegraphics[width=0.45\textwidth,trim = 0mm 0mm 0mm 0mm, clip]{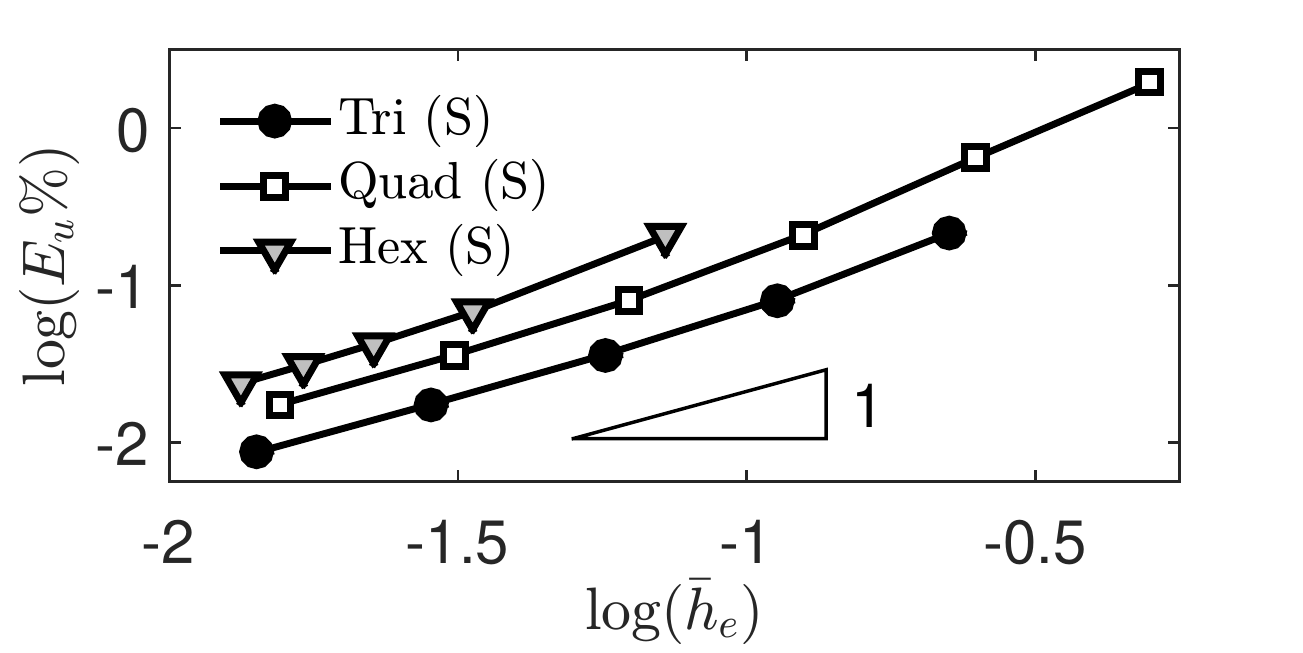}}
	\subfigure[]{\includegraphics[width=0.45\textwidth,trim = 0mm 0mm 0mm 0mm, clip]{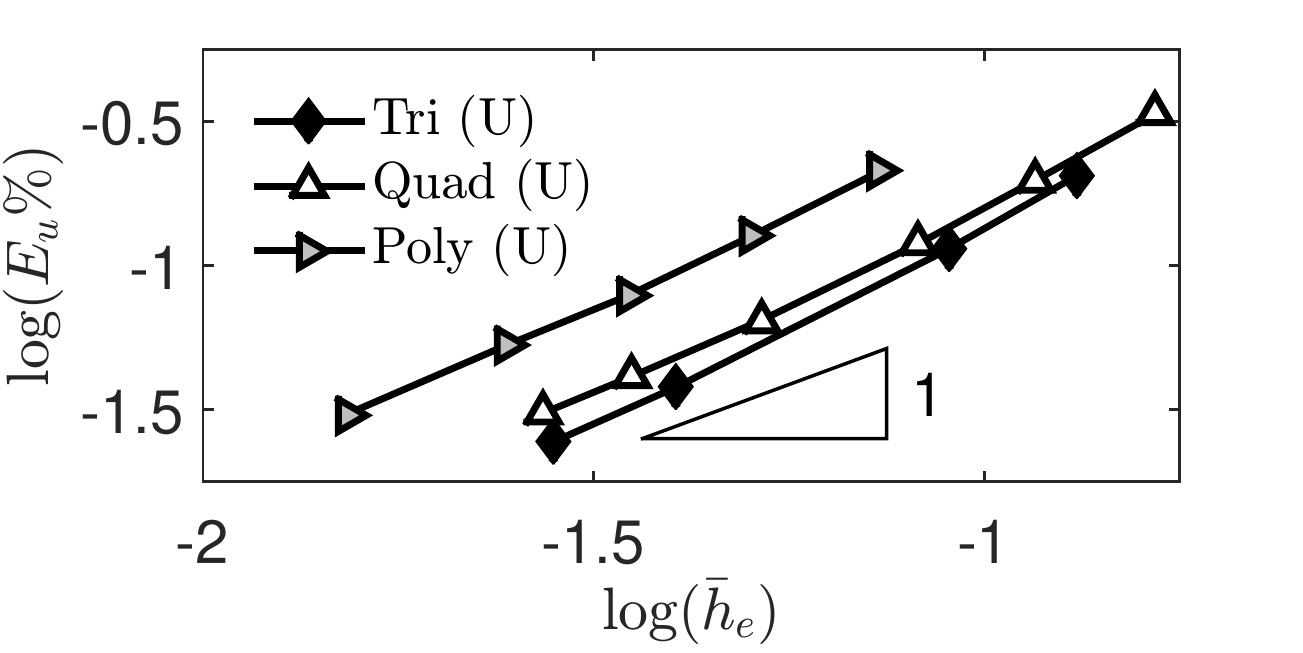}}\\
	\caption{$\bar{h}_e-$convergence results for Test $b$ on structured and unstructured meshes: (a) and (b) $E_{\bfsigma}$ error norm plots, (c) and (d) $E_{\bfsigma, \bdiv}$ error norm plots, (e) and (f) $E_{\bbu}$ error norm plots.}
	\label{fig:resuTestB}
\end{figure}

\subsubsection{Nearly incompressibility regime}\label{ss:near_inc}
A problem on the unit square domain $\Omega = [0, 1]^2$, with known analytical solution, is considered. A nearly incompressible material is chosen by selecting Lam\'e constants as $\lambda = 10^5$, $\mu = 0.5$.
The test is designed by choosing a required solution for the displacement field and deriving the load $\bbf$ accordingly. The displacement solution is as follows:

\begin{eqnarray}
	\label{eq:test_inc_sol}
	\left\{
	\begin{array}{l}
		u_1 = 0.5(\sin(2\pi x))^2\sin(2\pi y)\cos(2\pi y) \\
		u_2 = -0.5(\sin(2\pi y))^2\sin(2\pi x)\cos(2\pi x) .
	\end{array}
	\right .
\end{eqnarray}

Figure \ref{fig:resuTestC} reports the results obtained for both structured and unstructured meshes. In can be clearly seen that the proposed method shows the expected asymptotic rate of convergence also in this case.

\begin{figure}[h!]
	\centering
	\subfigure[]{\includegraphics[width=0.45\textwidth,trim = 0mm 0mm 0mm 0mm, clip]{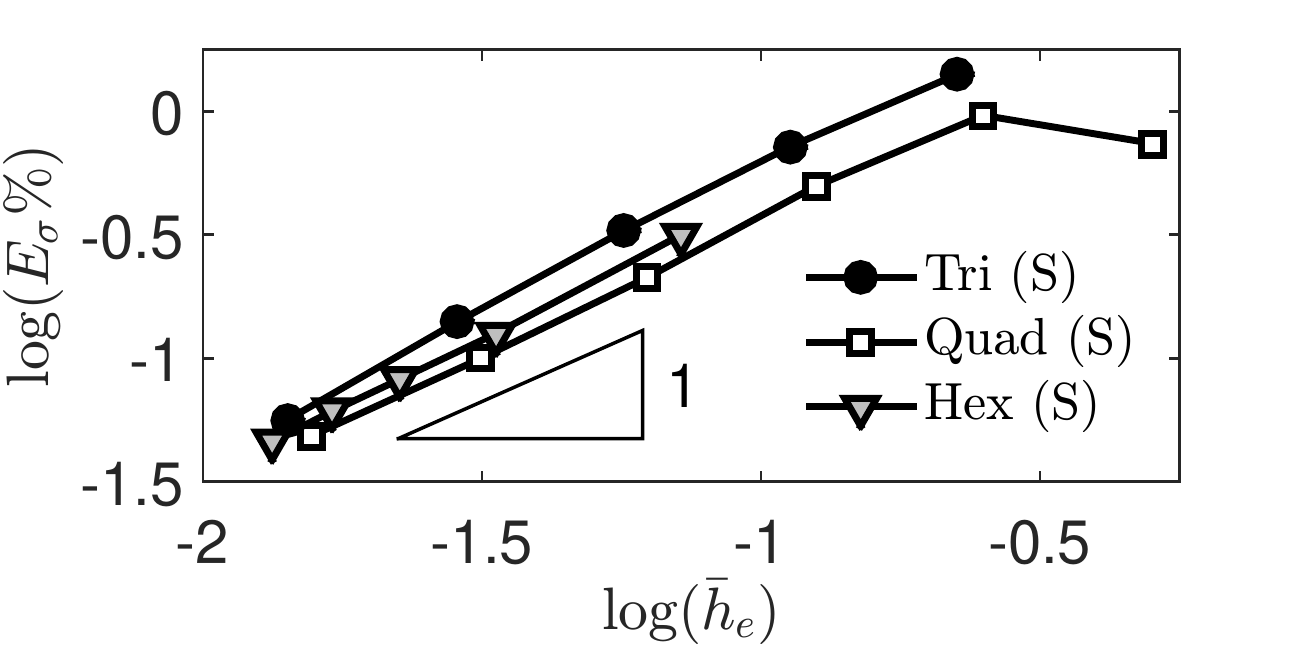}}
	\subfigure[]{\includegraphics[width=0.45\textwidth,trim = 0mm 0mm 0mm 0mm, clip]{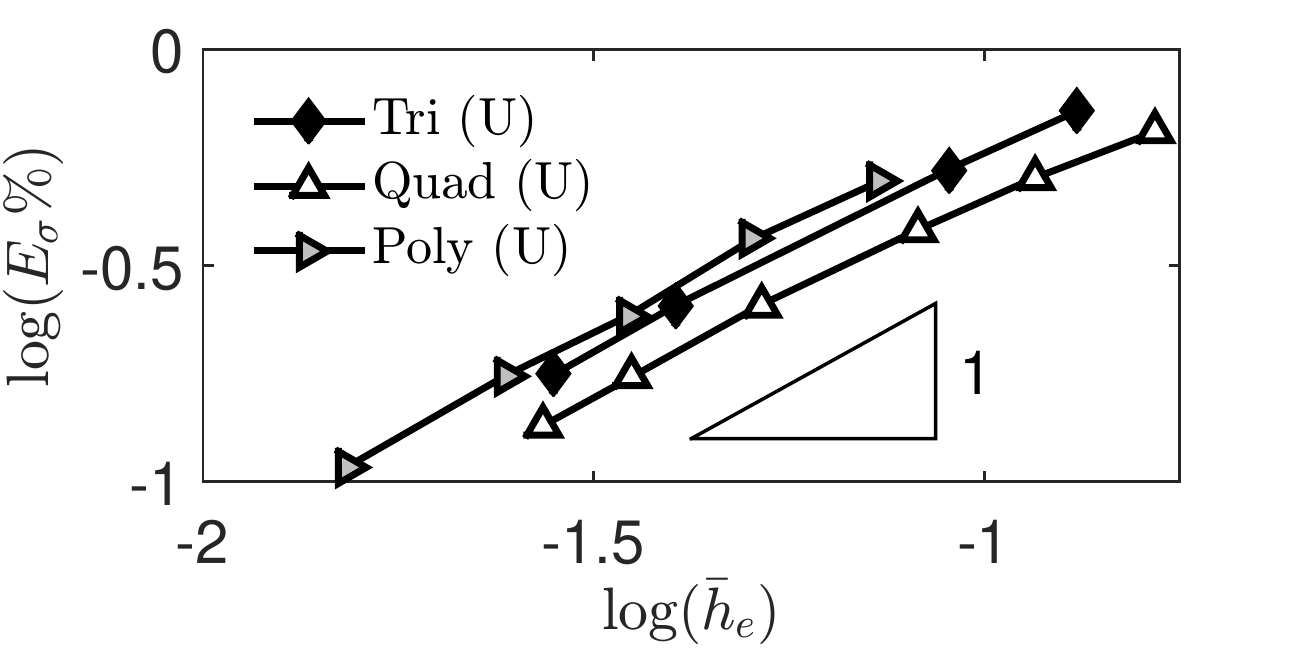}}\\
	\subfigure[]{\includegraphics[width=0.45\textwidth,trim = 0mm 0mm 0mm 0mm, clip]{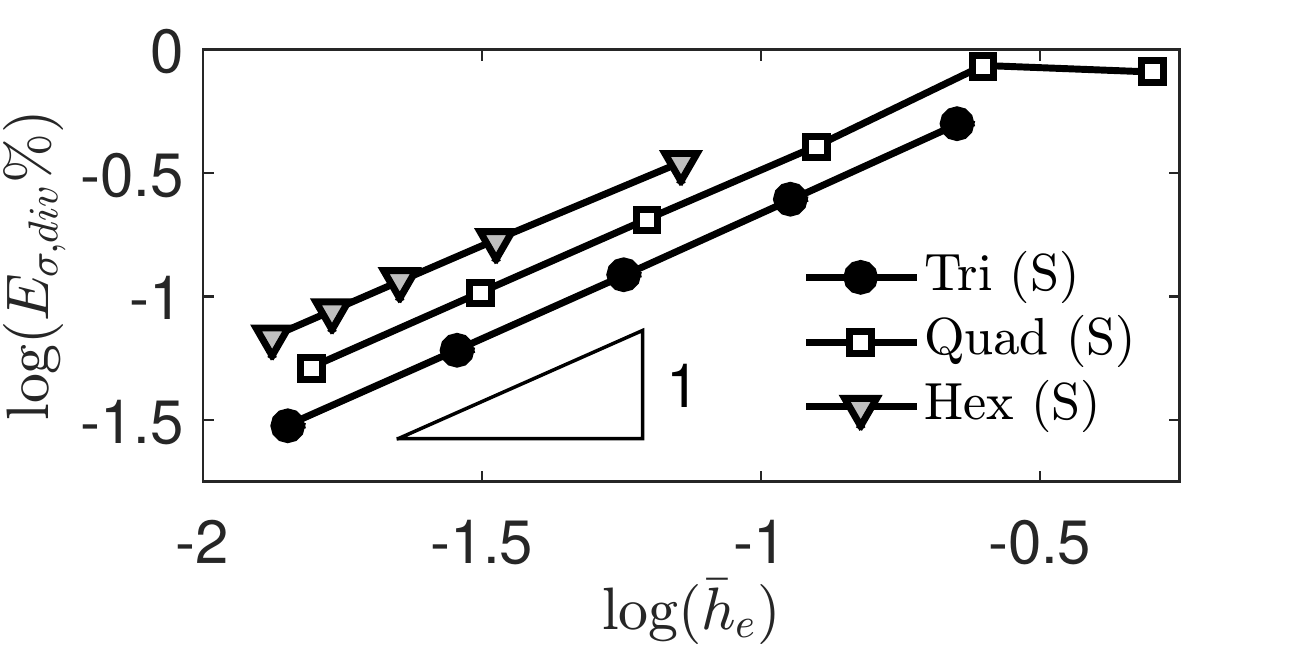}}
	\subfigure[]{\includegraphics[width=0.45\textwidth,trim = 0mm 0mm 0mm 0mm, clip]{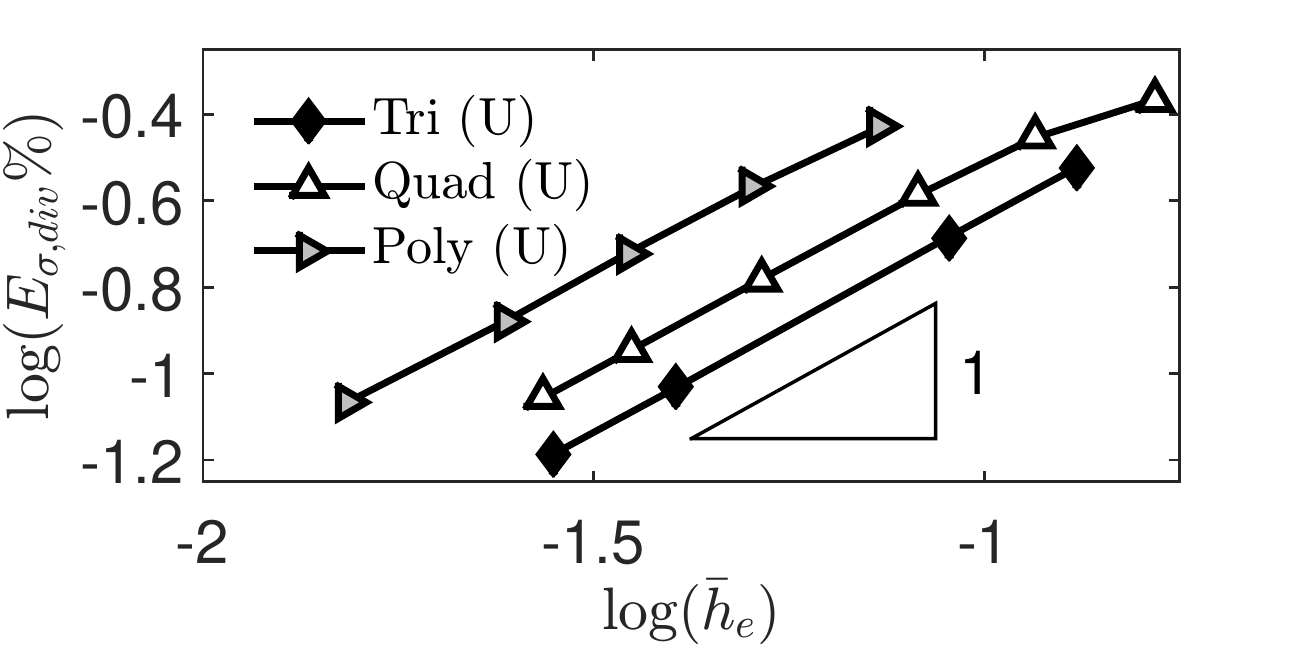}}\\
	\subfigure[]{\includegraphics[width=0.45\textwidth,trim = 0mm 0mm 0mm 0mm, clip]{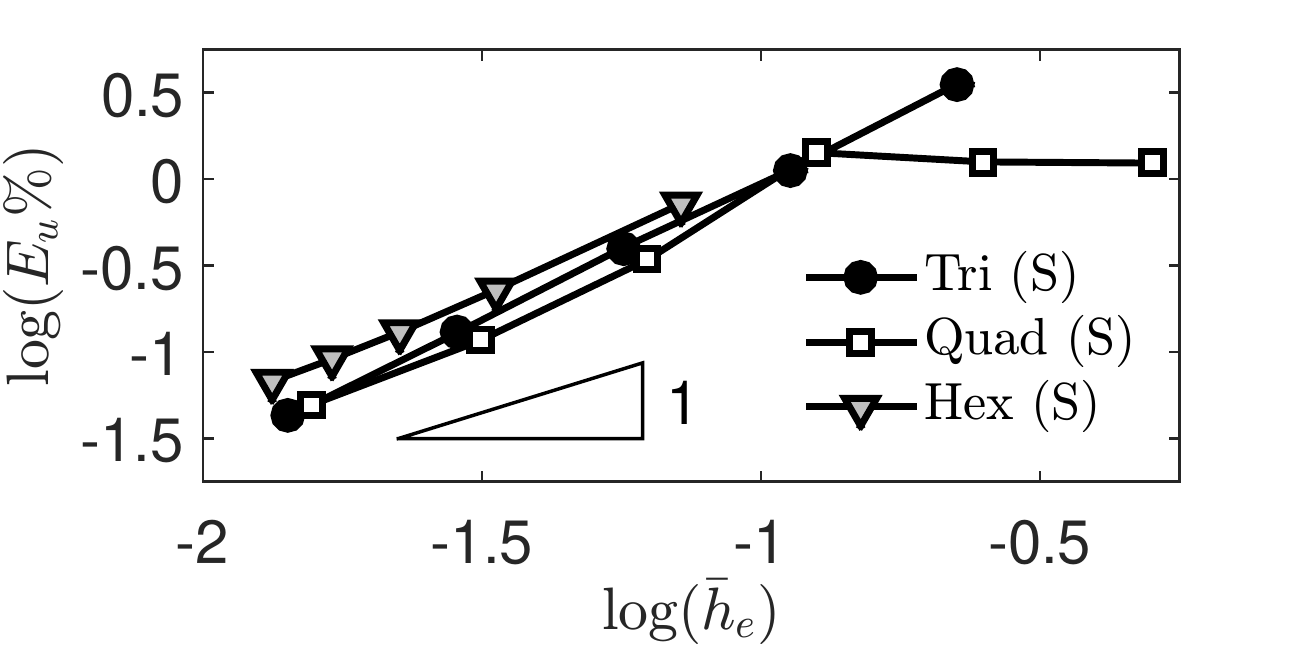}}
	\subfigure[]{\includegraphics[width=0.45\textwidth,trim = 0mm 0mm 0mm 0mm, clip]{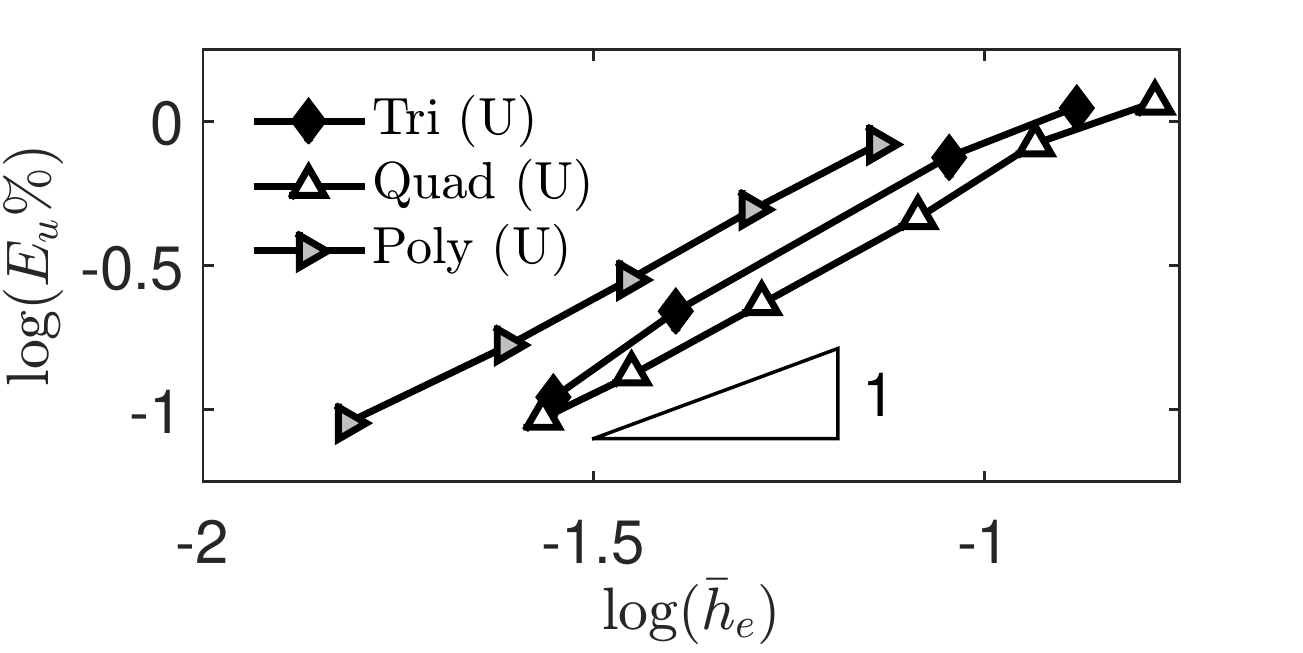}}\\
	\caption{Results for the nearly incompressible test on structured and unstructured meshes: (a) and (b) convergence of $E_{\bfsigma}$, (c) and (d) convergence of $E_{\bfsigma, \bdiv}$, (e) and (f) convergence for $E_{\bbu}$.}
	\label{fig:resuTestC}
\end{figure}

\subsection{Structural analysis benchmark: Cook's membrane}
\label{ss:cook}
The present section deals with the classical Cook's membrane 2D problem \cite{Zienckiewicz_Taylor2000}. The geometry of the domain $\O$ is presented in Fig. \ref{fig:Test_4_geom} with length data $H_1 = 44$, $H_2 = 16$, $L = 48$. The loading is given by a constant tangential traction $q = 6.25$ on the right edge of the domain. The Young modulus, $E$, is set equal to $70$ and two Poisson ratios are considered, one corresponding to $\nu = 1/3$ and one corresponding to a nearly incompressible case characterized by $\nu = 0.499995$.

\begin{figure}
	\begin{center}
		\includegraphics[bb = 160 18 600 580, scale=0.35]{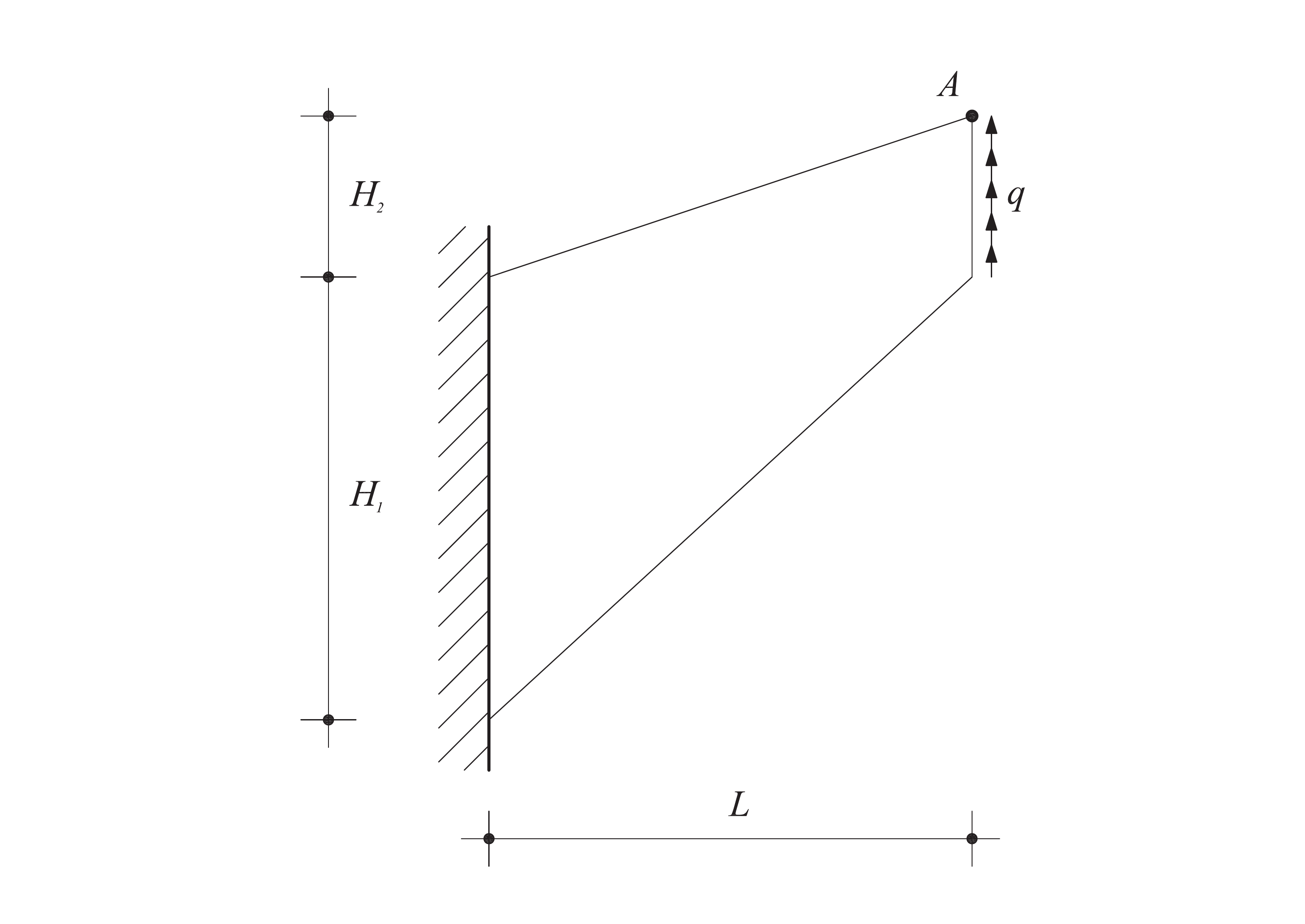}
	\end{center}
	\caption{
		Cook's membrane. Geometry, loading and boundary conditions.}
	\label{fig:Test_4_geom}
\end{figure}

The problem is solved using three types of meshes: an evenly distributed quadrilateral mesh denoted as Quad, a centroid based Voronoi tessellation, denoted as CVor, and a random based Voronoi tessellation indicated as RVor. An overview of the adopted meshes is reported in Fig. \ref{fig:cookMeshes}.
\begin{figure}[h!]
	\centering
	\renewcommand{\thesubfigure}{}
	\subfigure[Quad]{\includegraphics[width=0.32\textwidth,trim = 30mm 0mm 30mm 0mm, clip]{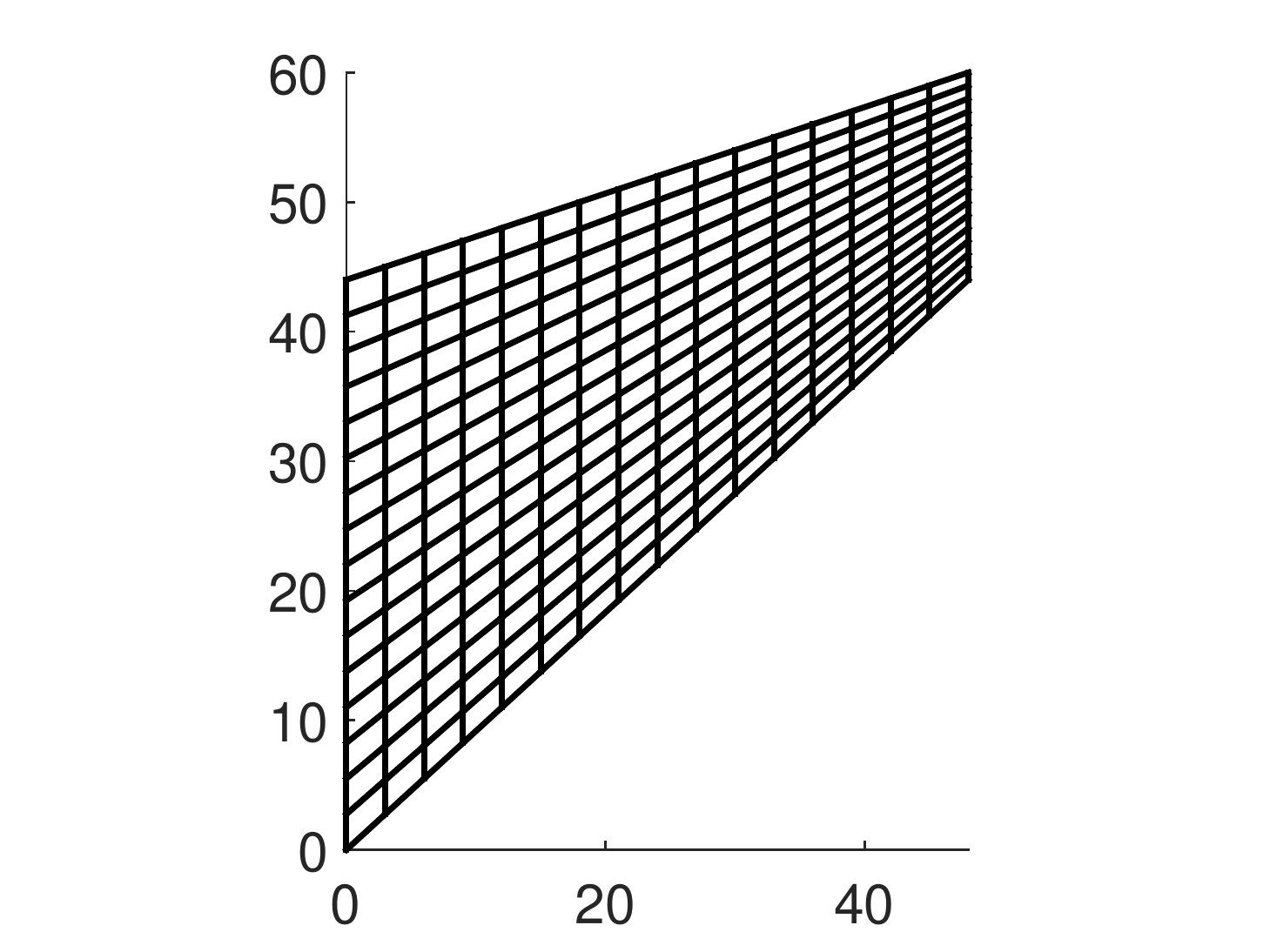}}
	\subfigure[CVor]{\includegraphics[width=0.32\textwidth,trim = 30mm 0mm 30mm 0mm, clip]{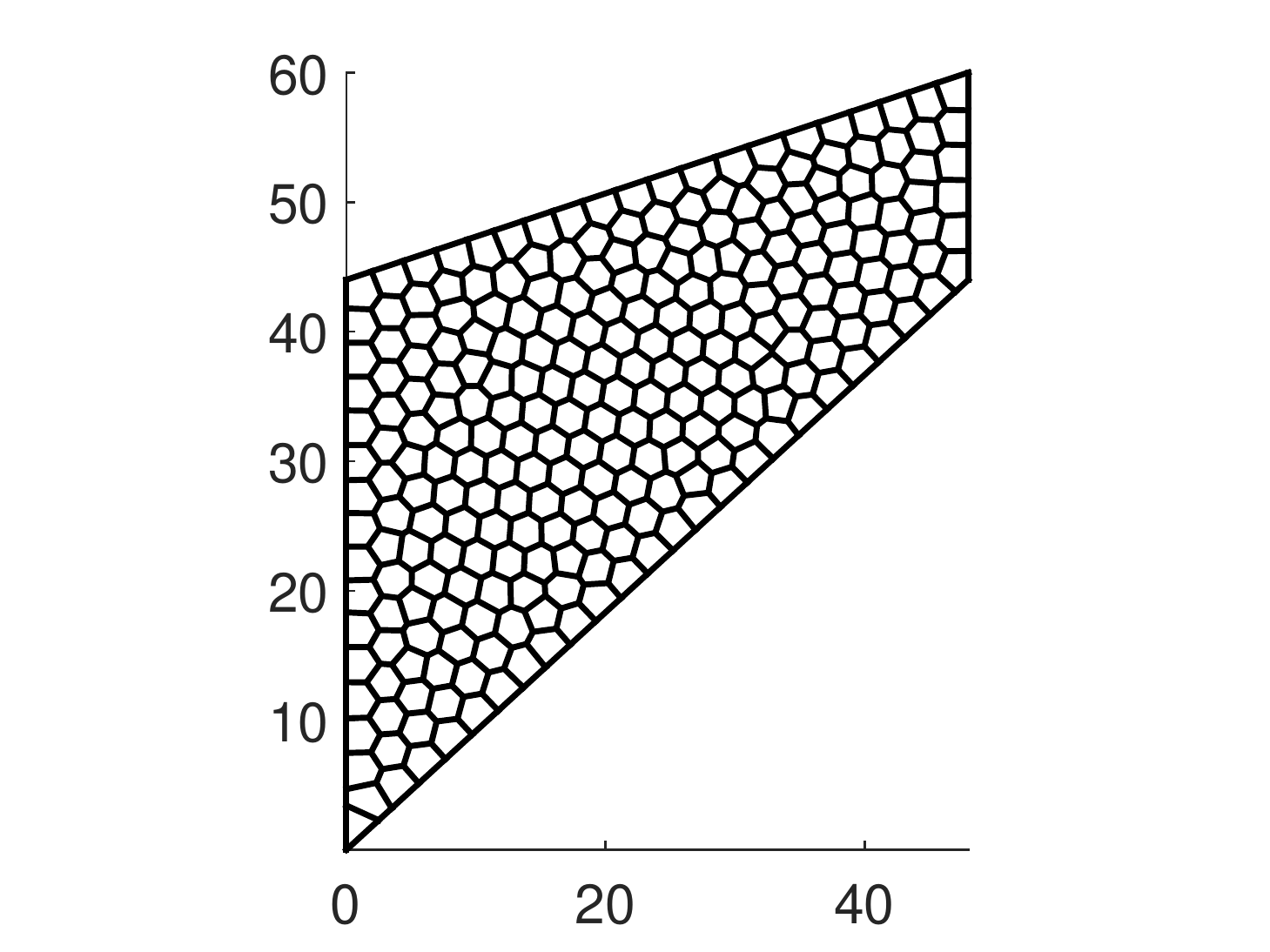}}
	\subfigure[RVor]{\includegraphics[width=0.32\textwidth,trim = 30mm 0mm 30mm 0mm, clip]{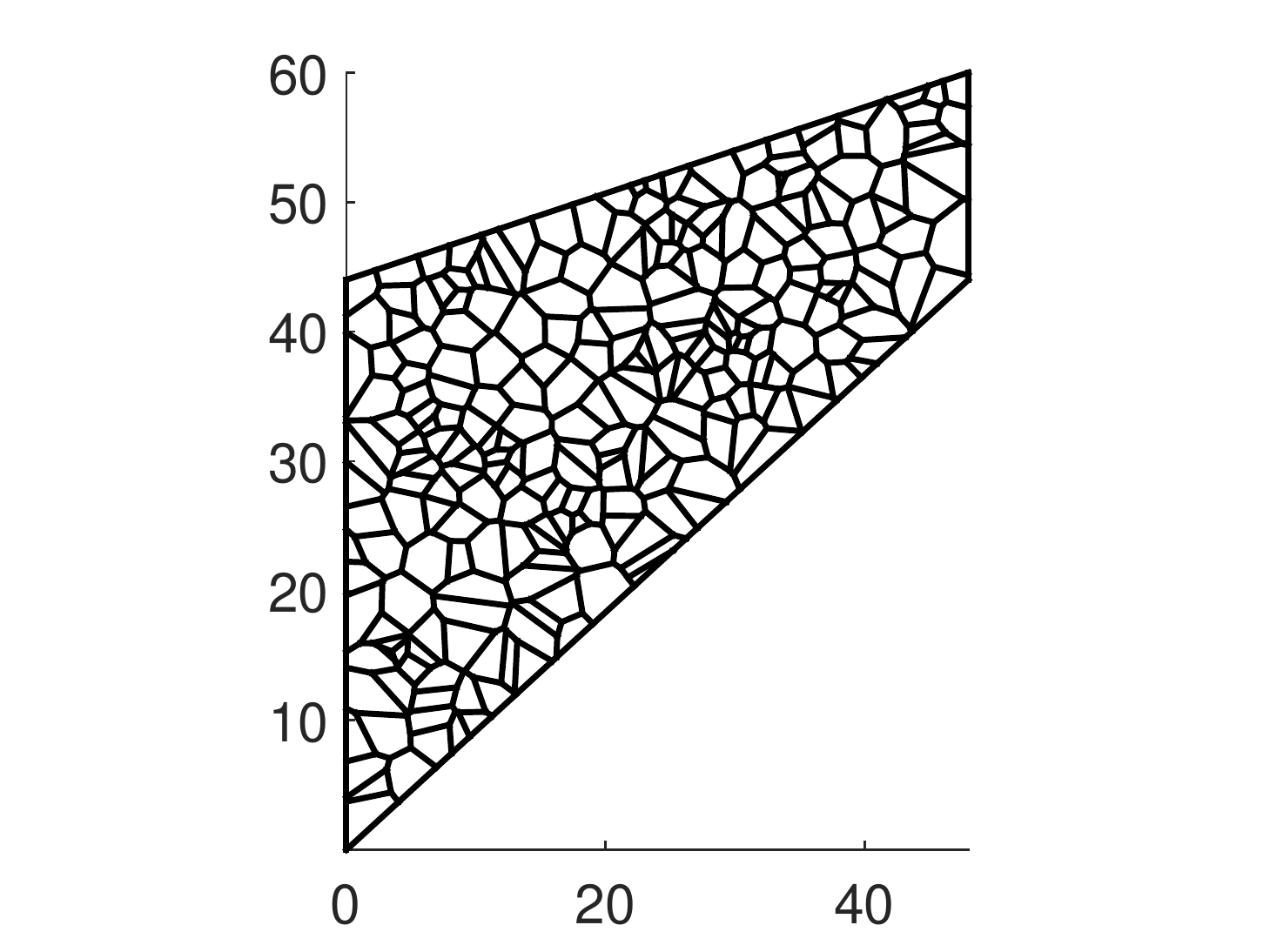}}
	\caption{Cook's membrane. Examples of the adopted meshes.}
	\label{fig:cookMeshes}
\end{figure}

Convergence results are reported in terms of mesh refinement monitoring $v_A$, the vertical displacement of point A (see Fig. \ref{fig:Test_4_geom}), approximated as the vertical displacement at the centroid of the closest polygon.
In particular, Fig. \ref{fig:cookConv}(a) corresponds to the case in which $\nu = 1/3$ while Fig. \ref{fig:cookConv}(b) reports the results obtained for the nearly incompressible case. The reference solution is indicated with a dotted red line corresponding to an overkilling accurate solution obtained with the hybrid-mixed CPE4I element \cite{ABAQUS:2011}. In accordance with the results of Section \ref{ss:near_inc}, it can be clearly observed that the proposed formulation is robust with respect to the compressibility parameter, as the convergence behaviour of both cases (a) and (b) is almost the same.

\begin{figure}[h!]
	\centering
	\subfigure[]{\includegraphics[width=0.48\textwidth,trim = 0mm 0mm 10mm 0mm, clip]{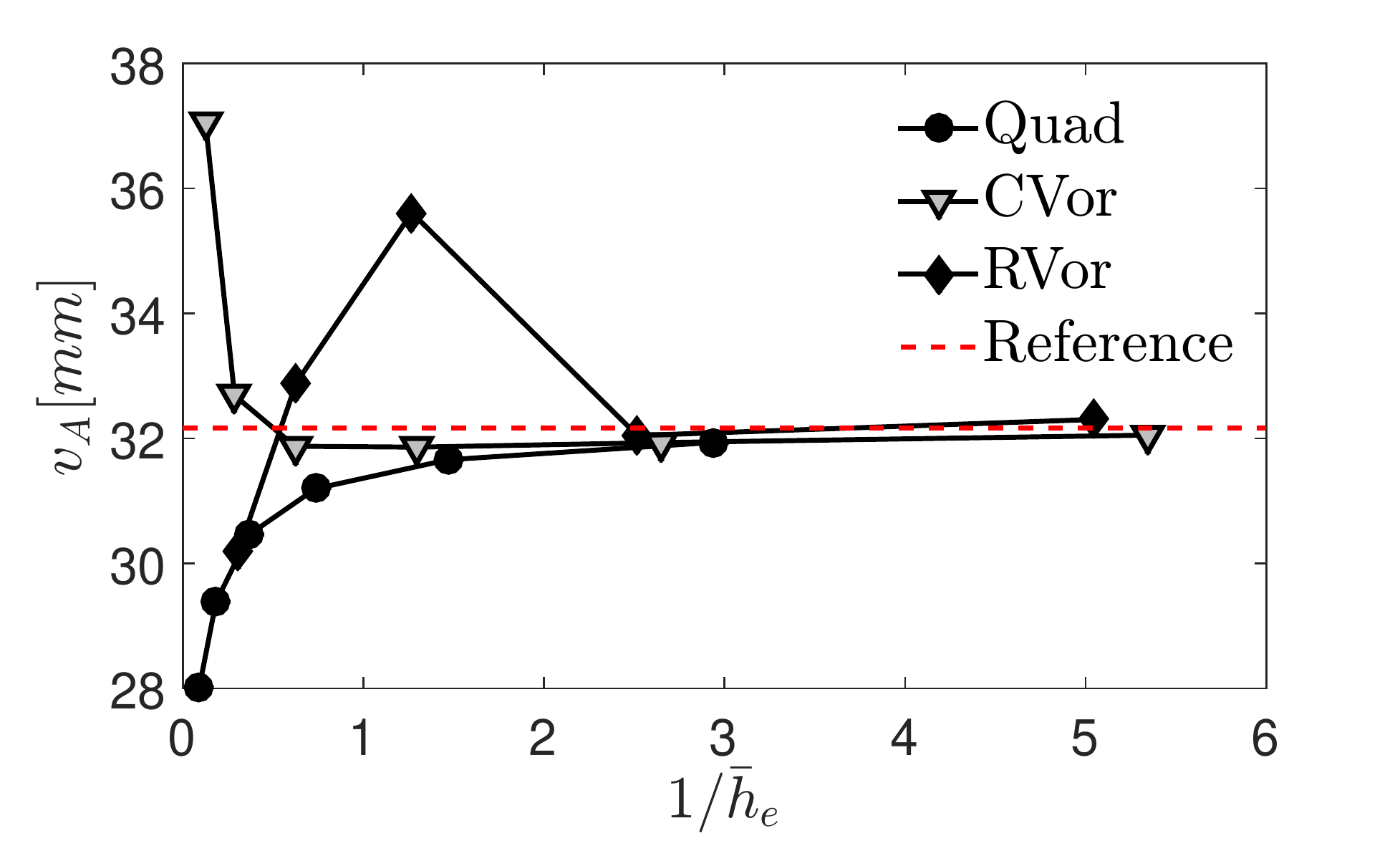}}
	\subfigure[]{\includegraphics[width=0.48\textwidth,trim = 0mm 0mm 10mm 0mm, clip]{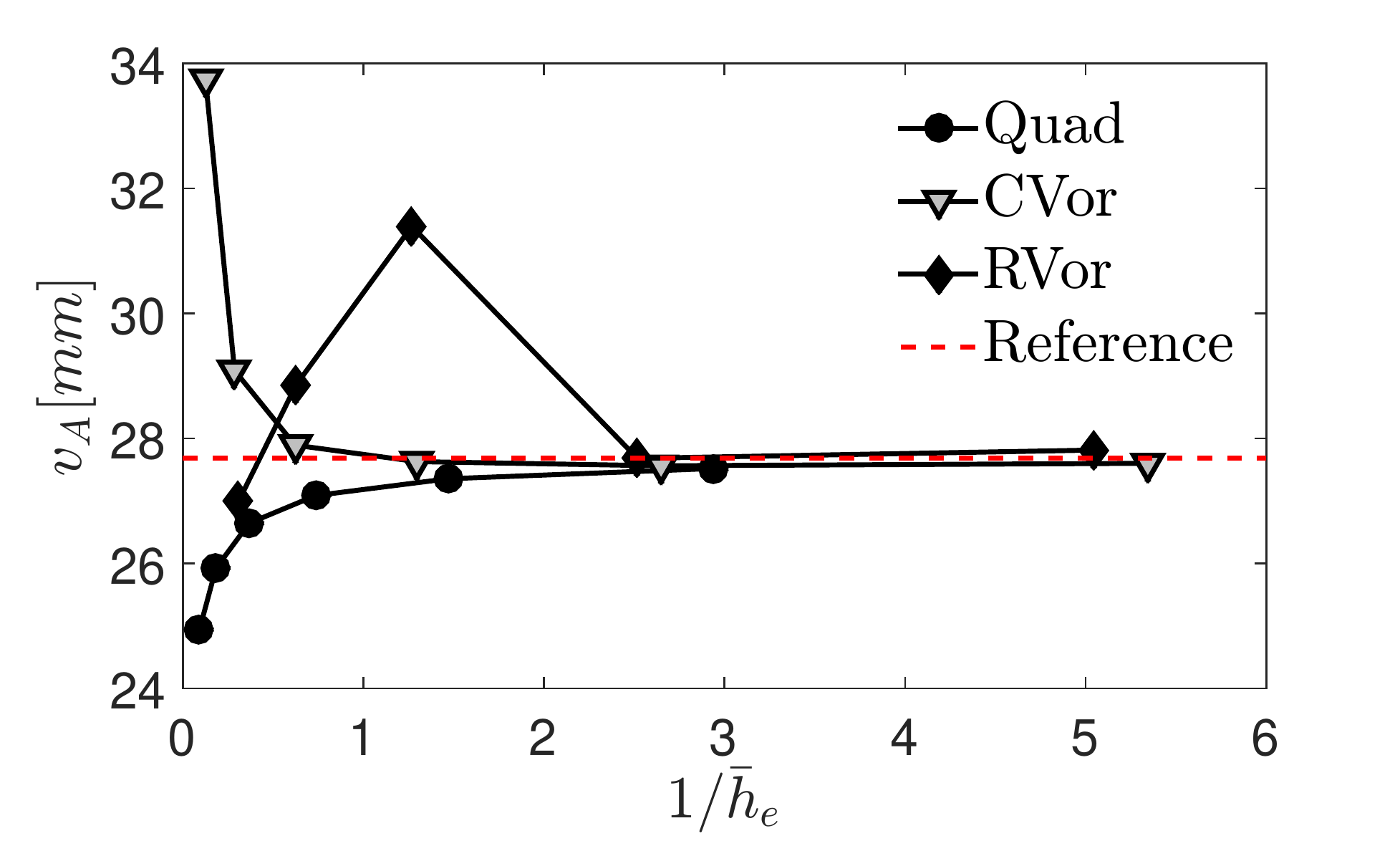}}
	\caption{Convergence of the tip vertical displacement $v_A$: (a) $\nu = 1/3$ and (b) $\nu = 0.499995$.}
	\label{fig:cookConv}
\end{figure}

Finally, contours representing the von Mises equivalent stress distributions are reported in Fig. \ref{fig:cookConv}. We remark that, inside the polygons, the stress distribution $\bfsigma_h$ is not known, but its projection $\Pi_{E}\bfsigma_h$ onto the constant tensors is (cf. \eqref{eq:proj}). Thus, we have used this latter quantity to compute the von Mises equivalent stress displayed in Fig. \ref{fig:cookConv}.
Finally, the results refer to the case $\nu = 1/3$, being the nearly incompressible case extremely similar.

\begin{figure}[h!]
	\centering
	\subfigure[]{\includegraphics[height=0.38\textwidth,trim = 20mm 0mm 43mm 0mm, clip]{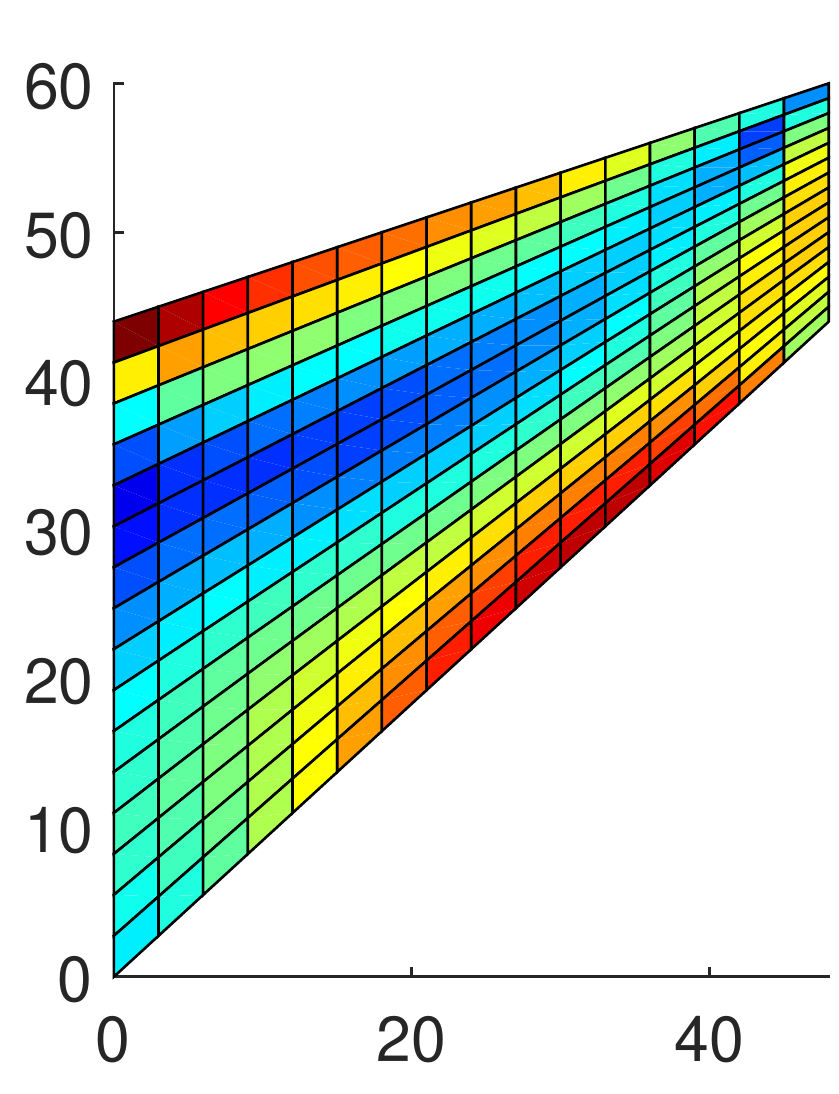}}
	\subfigure[]{\includegraphics[height=0.38\textwidth,trim = 20mm 0mm 43mm 0mm, clip]{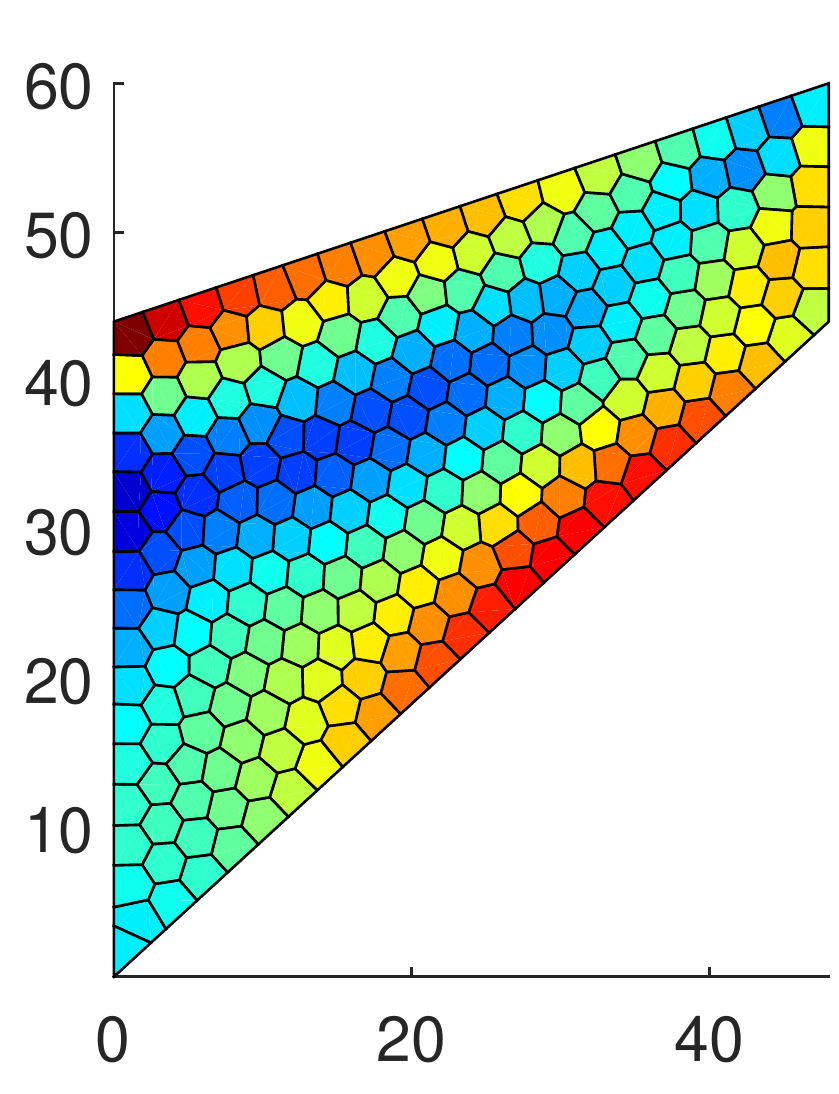}}
	\subfigure[]{\includegraphics[height=0.38\textwidth,trim = 20mm 0mm 20mm 0mm, clip]{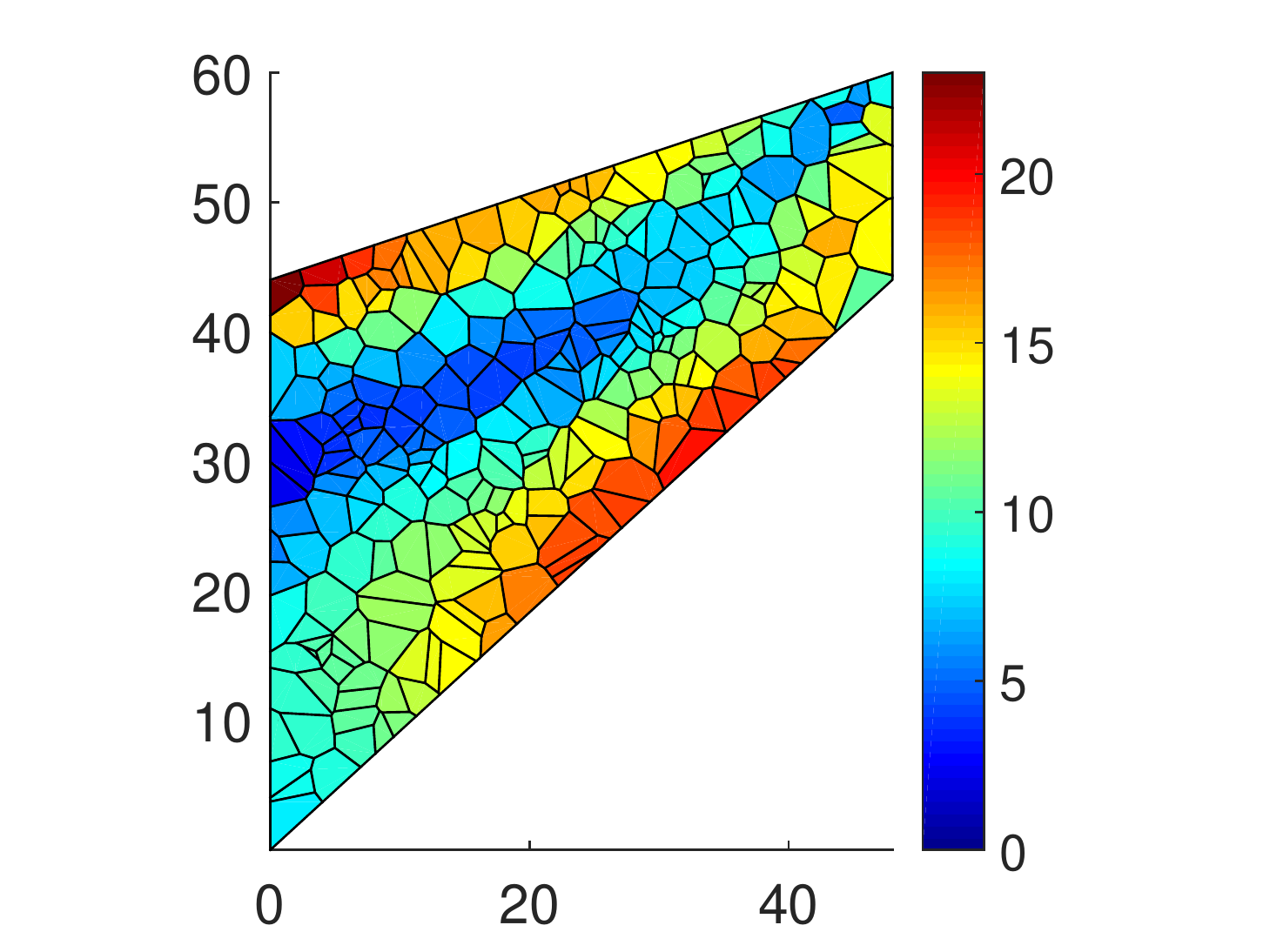}}
	\caption{Contours representing the von Mises equivalent stress distributions for $\nu = 1/3$: (a) Quad, (b) CVor, (c) RVor.}
	\label{fig:cookMises}
\end{figure}



\section{Stability and convergence analysis}\label{s:theoretical}

In this section, we provide a rigorous analysis of the proposed VEM method.
For all $E\in \Th$, we first introduce the space:

\begin{equation}\label{eq:sigma-comp}
\widetilde\Sigma(E):=\left\{ \bftau\in H(\bdiv;E)\ : \ \exists \bbw\in H^1(E)^2 \mbox{ \rm such that } \bftau=\C\teps(\bbw) \right\} .
\end{equation}

The global space $\widetilde\Sigma$ is defined as

\begin{equation}\label{eq:sigma-glob}
\widetilde\Sigma:=\left\{ \bftau\in H(\bdiv;\Omega)\ : \ \exists \bbw\in H^1(\O)^2 \mbox{ \rm such that } \bftau=\C\teps(\bbw) \right\} .
\end{equation}

In the sequel, given a measurable subset $\omega\subseteq \Omega$ and $r > 2$, we will use the space

\begin{equation}\label{eq:regspace}
W^r(\omega):=\left\{  \bftau  \ : \bftau\in L^r(\omega)^{2\times 2}_s \ , \ \bdiv\bftau\in L^2(\omega)^2    \right\} ,
\end{equation}
equipped with the obvious norm. Under our assumptions on the mesh, we recall the following version of the Korn's inequality:

\begin{equation}\label{eq:korn}
\inf_{\bbr\in RM(E)}\left(  h_E^{-1}||\bbv-\bbr ||_{0,E} + |\bbv-\bbr |_{1,E}  \right) \lesssim ||\teps(\bbv) ||_{0,E} \qquad \forall \bbv \in H^1(E)^2 .
\end{equation}
Given $\bbv\in H^1(E)^2$, the above inequality can be derived by classical results (see \cite{Oleinik-Kondratiev}, for instance), and by choosing $\bbr_\bbv\in RM(E)$ such that $\int_E(	\bbv -\bbr_\bbv)=\bfzero$.

We will also use the following result.

\begin{lem}\label{lem:aux} Suppose that assumptions $\mathbf{(A1)}$ and $\mathbf{(A2)}$ are fulfilled.
Given $E\in \Th$, let $\bbw\in H^1(E)^2$ be a solution of the problem:

\begin{equation}\label{eq:aux1}
\left\lbrace{
	\begin{aligned}
	&-\bdiv(\C\teps(\bbw))=\bbg\qquad & &\mbox{\rm in $E$}\\
	&(\C\teps(\bbw))\bbn =\bbh&& \mbox{\rm on $\partial E$} ,
	\end{aligned}
} \right.
\end{equation}
where $\bbg\in L^2(E)^2$ and $\bbh\in L^2(\partial E)^2$ satisfy the compatibility condition

\begin{equation}\label{eq:aux2}
\int_E \bbg\cdot\bbr + \int_{\partial E} \bbh\cdot\bbr = 0 \qquad \forall\bbr\in RM(E) .
\end{equation}
Then it holds:

\begin{equation}\label{eq:aux3}
||\C \teps(\bbw)||_{0,E}\lesssim h_E || \bbg||_{0,E} + h_E^{1/2} || \bbh||_{0,\partial E} .
\end{equation}
	
\end{lem}

\begin{proof}
For every $\bbr\in RM(E)$, we have

\begin{equation}\label{eq:aux4}
\begin{aligned}
|| \C\teps(\bbw)||_{0,E}^2 &\lesssim
\int_E \C\teps(\bbw) : \teps(\bbw)\\
&= \int_E \C\teps(\bbw) : \teps(\bbw - \bbr)
=\int_E \bbg\cdot ( \bbw-\bbr) + \int_{\partial E}\bbh\cdot (\bbw-\bbr) ,
\end{aligned}
\end{equation}
by which we get

\begin{equation}\label{eq:aux5}
\begin{aligned}
|| \C\teps(\bbw)||_{0,E}^2 \lesssim || \bbg ||_{0,E}\, || \bbw-\bbr ||_{0,E} + || \bbh ||_{0,\partial E}\, || \bbw-\bbr ||_{0,\partial E} .
\end{aligned}
\end{equation}
Under assumptions $\mathbf{(A1)}$ and $\mathbf{(A2)}$, the Agmon's inequality then gives

\begin{equation}\label{eq:aux6}
\begin{aligned}
||\C \teps(\bbw)||_{0,E}^2 & \lesssim || \bbg ||_{0,E}\, || \bbw-\bbr ||_{0,E}\\
&+ || \bbh ||_{0,\partial E} \left( h_E^{-1/2}|| \bbw-\bbr ||_{L^2( E)} +  h_E^{1/2}| \bbw-\bbr |_{H^1( E)}  \right).
\end{aligned}
\end{equation}
Estimate \eqref{eq:aux3} now follows from \eqref{eq:korn}.

\end{proof}

\subsection{An interpolation operator for stresses}\label{ss:interpol-oper}

We now introduce the local interpolation operator $\IE : W^r(E)\to  \Sigma_h(E)$, defined as follows. Given $\bftau\in W^r(E)$, $\IE\bftau\in \Sigma_h(E)$ is determined by:

\begin{equation}\label{eq:loc-interp_0}
\int_{\partial E} (\IE \bftau) \bbn\cdot \bfvarphi_\ast = \int_{\partial E}  \bftau\bbn\cdot \bfvarphi_\ast \qquad \forall \bfvarphi_\ast\in  R_\ast(\partial E) ,
\end{equation}
where

\begin{equation}\label{eq:Rast}
R_\ast(\partial E) = \left\{
\bfvarphi_\ast\in L^2(\partial E)^2 \,: \,
\bfvarphi_{\ast | e} = \bfgamma_e + \delta_e (\bbx -\bbx_C)^\perp \quad \bfgamma_e\in\R^2, \  \delta_e\in\R, \ \forall e\in\partial E  \right\}.
\end{equation}
If $\bftau$ is not sufficiently regular, the integral in the right-hsnd side of \eqref{eq:loc-interp_0} is intended as a duality between $W^{-\frac{1}{r},r}(\partial E)^2$ and $W^{\frac{1}{r},r'}(\partial E)^2$. If $\bftau$ is a regular function, the above condition is equivalent to require:

\begin{equation}\label{eq:loc-interp}
\left\lbrace{
\begin{aligned}
&\int_e (\IE \bftau) \bbn = \int_e  \bftau\bbn \qquad \forall e\in \partial E ;\\
&\int_e (\IE \bftau) \bbn\cdot (\bbx - \bbx_C)^\perp = \int_e  \bftau\bbn \cdot (\bbx - \bbx_C)^\perp \qquad \forall e\in \partial E .
\end{aligned}
} \right.
\end{equation}

The following result shows, in particular, that $\IE\bftau\in\Sigma_h(E)$ is well-defined by conditions \eqref{eq:loc-interp_0}.

\begin{lem}\label{lm:local-inter-well}
If $\bftau_h\in\Sigma_h(E)$, then

\begin{equation}\label{eq:interpol_cond}
\int_{\partial E} \bftau_h \bbn\cdot \bfvarphi_\ast = 0  \qquad \forall \bfvarphi_\ast\in  R_\ast(\partial E)
\end{equation}

imply $\bftau_h=\bfzero$.

\end{lem}

\begin{proof}
First, recall that for $\bftau_h\in\Sigma_h(E)$ it holds $(\bftau_h\bbn)_{|e}= \bbc_e + d_e s\,\bbn$ for each edge $e\in\partial E$, cf. \eqref{eq:edge_approx} and \eqref{eq:local_stress}. 	
By \eqref{eq:interpol_cond}, choosing $\bfvarphi_\ast$ such that $\bfvarphi_{\ast|e}=\bfgamma_e$ for each $e\in\partial E$, it follows that $(\bftau_h\bbn)_{|e}= d_e s\,\bbn$. Choosing now $\bfvarphi_{\ast|e}=\delta_e (\bbx -\bbx_C)^\perp$, conditions \eqref{eq:interpol_cond} then give

\begin{equation}\label{eq:interpol_cond2}
d_e\int_{ e} s\,\bbn\cdot (\bbx -\bbx_C)^\perp = 0  \qquad \forall e\in\partial E.
\end{equation}

A direct computation (for instance by using the Cavalieri-Simpson rule) shows that \eqref{eq:interpol_cond2} is equivalent to

\begin{equation}\label{eq:interpol_cond3}
d_e\frac{|e|}{12}\,\bbn\cdot(\bbq_e-\bbp_e)^\perp = 0  \qquad \forall e\in\partial E.
\end{equation}
Above, $\bbp_e$ and $\bbq_e$ denote the endpoints of $e$.
From \eqref{eq:interpol_cond3} we infer $d_e=0$ for each $e\in\partial E$, which concludes the proof.
\end{proof}

The global interpolation operator $\Ih : W^r(\O)\to \Sigma_h$ is then defined by simply glueing the local contributions provided by $\IE$. More precisely, we set $(\Ih\tau)_{|E} :=\IE\bftau_{|E}$ for every $E\in\Th$ and $\bftau\in W^r(\O)$.

\subsection{Approximation estimates}\label{ss:approx}

\begin{prop}\label{pr:approxest} Under assumptions $\mathbf{(A1)}$ and $\mathbf{(A2)}$,
for the interpolation operator $\IE$ defined in \eqref{eq:loc-interp}, the following estimates hold:

\begin{equation}\label{eq:l2est}
|| \bftau -\IE\bftau||_{0,E}\lesssim h_E |\bftau|_{1,E} \qquad \forall \bftau\in  \widetilde\Sigma(E) \cap H^1(E)^{2\times 2}_s .
\end{equation}
	
\begin{equation}\label{eq:divest}
\begin{aligned}
|| \bdiv(\bftau -\IE\bftau)||_{0,E}  \lesssim h_E |\bdiv\bftau|_{1,E}   \ \
\forall \bftau\in \widetilde\Sigma(E) \cap H^1(E)^{2\times 2}_s
\mbox{ \rm s.t.  $\bdiv\bftau\in H^1(E)^2$}.
\end{aligned}
\end{equation}	
	
\end{prop}

\begin{proof}
	
Let $ \bftau\in  \widetilde\Sigma(E) \cap H^1(E)^{2\times 2}_s $, and let $\bbw\in H^1(E)^2$ be such that $\bftau=\C\teps(\bbw)$, see \eqref{eq:sigma-comp}. Furthermore, consider $\IE \bftau\in\Sigma_h(E)$ and $\bbw^\ast\in H^1(E)^2$ such that $\IE\bftau=\C\teps(\bbw^\ast)$, see \eqref{eq:local_stress}.
Hence, setting $\bfdelta := (\bbw - \bbw^\ast)\in H^1(E)^2$, it holds:

\begin{equation}\label{eq:differ}
\bftau-\IE\bftau = \C \teps(\bfdelta) .
\end{equation}

Furthermore, using \eqref{eq:loc-interp}, \eqref{eq:div1} and \eqref{eq:div2}, we infer that $\bfdelta\in H^1(E)^2$ satisfies:

\begin{equation}\label{eq:problest}
\left\lbrace{
	\begin{aligned}
	&\bdiv(\C\teps(\bfdelta))=\bdiv\bftau - \frac{1}{|E|}\sum_{e\in\partial E}\int_e \bftau\bbn - \frac{(\bbx -\bbx_C)^\perp}{ \int_E | \bbx -\bbx_C |^2 }\sum_{e\in\partial E}\int_e \bftau\bbn\cdot (\bbx -\bbx_C)^\perp & &\mbox{\rm in $E$}\\
	&(\C\teps(\bfdelta))\bbn = \sum_{e\in\partial E} \left( \bftau\bbn -\frac{1}{|e|}\int_e \bftau\bbn\right)\chi_e && \mbox{\rm on $\partial E$} ,
	\end{aligned}
} \right.
\end{equation}
where $\chi_e$ denotes the characteristic function of the edge $e$.
Applying Lemma \ref{lem:aux} with:

\begin{equation}\label{eq:l2est2}
\left\lbrace{
	\begin{aligned}
	& \bbg :=  \frac{1}{|E|}\sum_{e\in\partial E}\int_e \bftau\bbn + \frac{(\bbx -\bbx_C)^\perp}{ \int_E | \bbx -\bbx_C |^2 }\sum_{e\in\partial E}\int_e \bftau\bbn\cdot (\bbx -\bbx_C)^\perp  - \bdiv\bftau \\
	& \bbh := \sum_{e\in\partial E} \left( \bftau\bbn - \frac{1}{|e|}\int_e \bftau\bbn\right)\chi_e ,
	\end{aligned}
} \right.
\end{equation}
we get

\begin{equation}\label{eq:l2est3}
|| \bftau -\IE\bftau ||_{0,E} = || \C\teps(\bfdelta) ||_{0,E} \lesssim   h_E || \bbg||_{0,E} + h_E^{1/2} || \bbh||_{0,\partial E} .
\end{equation}
We now estimate $\bbg$ and $\bbh$. We denote respectively with $\Pi_{0,E}$, $\Pi_{RM,E}$ and $\Pi_{0,\partial E}$ the $L^2$-projection operators onto the constant functions on $E$, onto the space $RM(E)$ (see \eqref{eq:rigid}), and on the piecewise constant functions on $\partial E$ (with respect to the edge subdivision of $\partial E$).

The divergence theorem and a direct computation show that:

\begin{equation}\label{eq:projRM}
\frac{1}{|E|}\sum_{e\in\partial E}\int_e \bftau\bbn + \frac{(\bbx -\bbx_C)^\perp}{ \int_E | \bbx -\bbx_C |^2 }\sum_{e\in\partial E}\int_e \bftau\bbn\cdot (\bbx -\bbx_C)^\perp = \Pi_{RM,E}\bdiv\bftau.
\end{equation}
Therefore, from the first equation of \eqref{eq:l2est2}, we have

\begin{equation}\label{eq:l2est4}
\bbg = \Pi_{RM,E} \bdiv\bftau - \bdiv\bftau  .
\end{equation}
Noting that $\P_0(E)^2\subset RM(E)$, from the properties of the $L^2$ projection operator, we then get

\begin{equation}\label{eq:l2est5}
|| \bbg||_{0,E} = || \Pi_{RM,E} \bdiv\bftau - \bdiv\bftau||_{0,E} \leq
|| \Pi_{0,E} \bdiv\bftau - \bdiv\bftau||_{0,E}
\lesssim
 || \bdiv\bftau ||_{0,E}
\end{equation}
and

\begin{equation}\label{eq:l2est5bis}
|| \bbg||_{0,E} = || \Pi_{RM,E} \bdiv\bftau - \bdiv\bftau||_{0,E} \leq
|| \Pi_{0,E} \bdiv\bftau - \bdiv\bftau||_{0,E}
\lesssim
h_E| \bdiv\bftau |_{1,E} .
\end{equation}

For the second equation of \eqref{eq:l2est2}, we remark that:

\begin{equation}\label{eq:proj0}
\bbh = \sum_{e\in\partial E} \left( \bftau\bbn - \frac{1}{|e|}\int_e \bftau\bbn\right)\chi_e =
\sum_{e\in\partial E} \left( \bftau - \frac{1}{|e|}\int_e \bftau\right)\bbn\chi_e =
\left( \bftau - \Pi_{0,\partial E}\bftau\right)\bbn .
\end{equation}
Hence, using a standard approximation estimate and a trace inequality, we get

\begin{equation}\label{eq:l2est9}
\begin{aligned}
|| \bbh ||_{0,\partial E} =
|| ( \bftau - \Pi_{0,\partial E}\bftau )\bbn ||_{0,\partial E}&\leq
|| \bftau - \Pi_{0,\partial E}\bftau||_{0,\partial E}
\lesssim h_E^{1/2} |\bftau |_{1/2,\partial E}\\
& \lesssim h_E^{1/2} |\bftau |_{1,E} .
\end{aligned}
\end{equation}
	
Taking into account \eqref{eq:l2est5} and \eqref{eq:l2est9}, from  \eqref{eq:l2est3} we obtain estimate \eqref{eq:l2est}.		
		
We now notice that from \eqref{eq:differ}, \eqref{eq:problest} and \eqref{eq:l2est2}, we have:

\begin{equation}\label{eq:divest1}
\bdiv(\bftau -\IE\bftau) = - \bbg .
\end{equation}		

Then, using \eqref{eq:l2est5bis}, we immediately get \eqref{eq:divest}.

\end{proof}


\subsection{Proving the {\em ellipticity-on-the-kernel} condition}\label{ss:elker}

We first notice that by \eqref{eq:proj}, \eqref{eq:ah1} and \eqref{eq:stab1}, using the techniques of \cite{volley,BFMXX}, one has:

\begin{equation}\label{eq:stabE}
||\bftau_h||_{0,E}^2\lesssim a_E^h(\bftau_h,\bftau_h)\lesssim ||\bftau_h||_{0,E}^2\qquad \forall \bftau_h\in\Sigma_h(E).
\end{equation}

We also notice that (see \eqref{eq:global-stress}, \eqref{eq:local_stress} and \eqref{eq:global-displ}, \eqref{eq:local_displ}):

\begin{equation}\label{eq:kern-incl}
\bdiv(\Sigma_h)\subseteq U_h .
\end{equation}
As a consequence, introducing the discrete kernel $K_h\subseteq \Sigma_h$:

\begin{equation}\label{eq:div_incl}
K_h =\{ \bftau_h\in\Sigma_h\, :\, (\bdiv \bftau_h,\bbv_h)=0 \quad \forall \bbv_h\in U_h  \},
\end{equation}
we infer that $\bftau_h\in K_h$ implies $\bdiv \bftau_h=\bfzero$. Hence, it holds:

\begin{equation}\label{eq:l2-hdiv}
|| \bftau_h ||_\Sigma = ||\bftau_h ||_0\qquad \forall \bftau_h\in K_h .
\end{equation}

We are now ready to prove the following {\em ellipticity-on-the-kernel} condition.

\begin{prop}\label{pr:elker}
For the method described in Section \ref{s:HR-VEM}, there exists a constant $\alpha >0$ such that

\begin{equation}\label{eq:elker}
a_h(\bftau_h,\bftau_h)\ge \alpha\, || \bftau_h||^2_\Sigma\qquad \forall \bftau_h\in K_h .
\end{equation}

\end{prop}

\begin{proof}
By recalling \eqref{eq:global-ah}, from \eqref{eq:stabE} we get the existence of $\alpha>0$ such that

\begin{equation}\label{eq:elker2}
a_h(\bftau_h,\bftau_h)\ge \alpha\, || \bftau_h||^2_0\qquad \forall \bftau_h\in \Sigma_h .
\end{equation}
Estimate \eqref{eq:elker} now follows by recalling \eqref{eq:l2-hdiv}.
\end{proof}

\begin{remark}\label{rm:incopmress_h}
Notice that for our method it holds $K_h\subset K$, where $K$ is defined by \eqref{eq:kernel}. Considering an isotropic material, see \eqref{eq:hom-iso}, from Remark \ref{rm:incompress} we infer that the coercivity constant $\alpha$ can be chosen independent of $\lambda$. Therefore, our numerical method does not suffer from volumetric locking (see \cite{Hughes:book}, for instance) and can be used also for nearly incompressible materials. This feature is confirmed by the numerical tests presented in Section \ref{s:numer}.
\end{remark}


\subsection{Proving the {\em inf-sup} condition}\label{ss:infsup}

We start by stating the following proposition, which can be derived by regularity results for the elasticity problem on Lipschitz domains (see \cite{HMW}, for example).

\begin{prop}\label{pr:reg-inf-sup}	
	Given the polygonal domain $\Omega$, there exist $s>2$ and $\beta^\ast>0$ such that
	
	\begin{equation}\label{eq:reg-inf-sup}
	\sup_{\bftau\in W^s(\O)}\frac{(\bdiv \bftau,\bbv)}{|| \bftau||_{W^s(\O)}}\ge \beta^\ast ||\bbv||_{0,\O}\qquad \forall\, \bbv\in L^2(\O)^2 ,
	\end{equation}
where $W^s(\O)$ is the Banach space defined by \eqref{eq:regspace}.	
\end{prop}

We are now ready to prove the discrete {\em inf-sup} condition for our choice of the approximation spaces.

\begin{prop}\label{pr:inf-sup} Suppose that  assumptions $\mathbf{(A1)}$ and $\mathbf{(A2)}$ are fulfilled. There exists $\beta>0$ such that
	
	\begin{equation}\label{eq:inf-sup}
	\sup_{\bftau_h\in \Sigma_h}\frac{(\bdiv \bftau_h,\bbv_h)}{|| \bftau_h||_{\Sigma}}\ge \beta ||\bbv_h||_{0,\O}\qquad \forall\, \bbv_h\in U_h .
	\end{equation}
	
\end{prop}

\begin{proof}
	
We will apply Fortin's criterion (see \cite{BoffiBrezziFortin}), using the operator $\Ih : W^s(\O)\to \Sigma_h$, see \eqref{eq:loc-interp} for the definition of the local contributions. More precisely, we will show that it holds:

\begin{equation}\label{eq:infsup1}
\left\lbrace{
	\begin{aligned}
	& \displaystyle{ \int_{\O}\bdiv(\Ih\bftau)\cdot \bbv_h= \int_{\O}\bdiv\bftau\cdot \bbv_h} \qquad \forall \bbv_h\in U_h \ , \ \forall \bftau \in  W^s(\O) ,\\
	& || \Ih \bftau||_{\Sigma} \lesssim || \bftau ||_{W^s(\O)}\qquad \forall \bftau \in  W^s(\O) .
	\end{aligned}
} \right.
\end{equation}
Together with \eqref{eq:reg-inf-sup}, conditions \eqref{eq:infsup1} imply \eqref{eq:inf-sup}, see \cite{BoffiBrezziFortin}.

To prove the first condition in \eqref{eq:infsup1}, recalling that $\bbv_{h|E}\in RM(E)$, it is sufficient to show that:

\begin{equation}\label{eq:infsup2}
 	\int_{E}\bdiv(\IE\bftau)\cdot\bbr= \int_{E}\bdiv\bftau\cdot\bbr \qquad \forall\bbr\in RM(E),\ \forall E\in \Th .
\end{equation}
The above equation directly follows from the divergence theorem and definition \eqref{eq:loc-interp_0}.

We now prove the continuity estimate (i.e. the second equation in \eqref{eq:infsup1}). We will exploit again Lemma \ref{lem:aux}. More precisely, we take $\bbw^\ast\in H^{1}(E)^2$ such that $ \IE\bftau = \C \teps(\bbw^\ast)$. It follows that $\bbw^\ast$ solves, cf. \eqref{eq:projRM}:

\begin{equation}\label{eq:infsup3}
\left\lbrace{
	\begin{aligned}
	&\bdiv(\C\teps(\bbw^\ast))=\Pi_{RM,E}\bdiv \bftau\\
	&(\C\teps(\bbw^\ast))\bbn = \sum_{e\in\partial E} \bbc_e \chi_e && \mbox{\rm on $\partial E$} ,
	\end{aligned}
} \right.
\end{equation}
where the $\bbc_e$'s are given by the dualities for the couple $< W^{-\frac{1}{s},s}(\partial E), W^{\frac{1}{s},s'}(\partial E)>$:

\begin{equation}\label{eq:infsup4}
\bbc_e :=  \frac{1}{|e|}\left( < \bftau\bbn ,  \chi_e \bbt > \bbt +   < \bftau\bbn , \chi_e \bbn >\bbn \right) .
\end{equation}

From \eqref{eq:infsup3}
we obviously deduce

\begin{equation}\label{eq:infsup5}
|| \bdiv (\IE\bftau)||_{0,E} = || \bdiv (\C\teps(\bbw^\ast)) ||_{0,E} =
 || \Pi_{RM,E}\bdiv \bftau ||_{0,E}  \leq
||\bdiv\bftau||_{0,E} .
\end{equation}

We now apply Lemma \ref{lem:aux} with:

\begin{equation}\label{eq:infsup6}
	 \bbg := - \Pi_{RM,E}\bdiv \bftau \ , \quad
	\bbh := \sum_{e\in\partial E} \bbc_e \chi_e ,
\end{equation}
and estimate $||\bbh||_{0,\partial E}$. We start by noting that:

\begin{equation}\label{eq:infsup7}
 ||\bbh||_{0,\partial E} =
 \left(   \sum_{e\in\partial E} |\bbc_e|^2  |e|  \right)^{1/2} \lesssim h_E^{1/2}
 \left(   \sum_{e\in\partial E} |\bbc_e|^2   \right)^{1/2}  .
\end{equation}
	
A duality estimate and a trace bound shows that

\begin{equation}\label{eq:infsup8}
< \bftau\bbn ,  \chi_e \bbt >\ \lesssim
|| \bftau\bbn ||_{W^{-\frac{1}{s},s}(\partial E)} || \chi_e \bbt ||_{W^{\frac{1}{s},s'}(\partial E)}
\lesssim || \bftau\bbn ||_{W^{-\frac{1}{s},s}(\partial E)}
\lesssim ||\bftau||_{W^s(E)} .
\end{equation}	
Similarly, it holds:

\begin{equation}\label{eq:infsup9}
< \bftau\bbn ,  \chi_e \bbn >\ \lesssim
 ||\bftau||_{W^s(E)} .
\end{equation}	

From \eqref{eq:infsup4}, \eqref{eq:infsup8} and \eqref{eq:infsup9} we get

\begin{equation}\label{eq:infsup10}
|\bbc_e| \lesssim h_E^{-1}
||\bftau||_{W^s(E)} ,
\end{equation}
by which we deduce, see \eqref{eq:infsup7}:

\begin{equation}\label{eq:infsup11}
||\bbh||_{0,\partial E} \lesssim
h_E^{-1/2}||\bftau||_{W^s(E)} .
\end{equation}

Lemma \ref{lem:aux} thus gives

\begin{equation}\label{eq:infsup12}
|| \IE\bftau||_{0,E} = ||\C \teps(\bbw^\ast) ||_{0,E}\lesssim ||\bftau||_{W^s(E)} .
\end{equation}
The continuity estimate in \eqref{eq:infsup1} now follows by collecting all the local estimates \eqref{eq:infsup12}.

\end{proof}

\subsection{Error estimates}\label{ss:errest}

We denote with $\P_0(\Th)$ the space of piecewise constant functions with respect to the given mesh $\Th$. We can prove the Proposition:

\begin{prop}\label{pr:error-est} Suppose that  assumptions $\mathbf{(A1)}$ and $\mathbf{(A2)}$ are fulfilled.
For every $(\bfsigma_I,\bbu_I)\in\Sigma_h\times U_h$ and every $\bfsigma_\pi\in \P_0(\Th)^{2\times 2}_s$, the following error equation holds:

\begin{equation}\label{eq:erroreq}
|| \bfsigma- \bfsigma_h||_\Sigma + || \bbu - \bbu_h ||_U \lesssim || \bfsigma- \bfsigma_I||_\Sigma + || \bbu - \bbu_I ||_U  + + h\,||\bdiv\bfsigma_I||_{0,\O} + || \bfsigma- \bfsigma_\pi||_{0,\O}
.
\end{equation}

\end{prop}

\begin{proof}
Given $(\bfsigma_I,\bbu_I)\in\Sigma_h\times U_h$, we form $(\bfsigma_h - \bfsigma_I,\bbu_h - \bbu_I)\in\Sigma_h\times U_h$. Then, using the
{\em ellipticity-on-the-kernel} condition of Proposition \ref{pr:elker} and the {\em inf-sup} condition of Proposition \ref{pr:inf-sup}, there exists $(\bftau_h, \bbv_h)\in\Sigma_h\times U_h$ such that (see \cite{BoffiBrezziFortin} and \cite{Braess:book}, for instance):

\begin{equation}\label{eq:costest}
|| \bftau_h||_\Sigma + || \bbv_h ||_U
\lesssim 1
\end{equation}
and

\begin{equation}\label{eq:stabest}
||\bfsigma_h - \bfsigma_I||_\Sigma + || \bbu_h - \bbu_I ||_U
\lesssim \A_h (\bfsigma_h - \bfsigma_I,\bbu_h - \bbu_I;\bftau_h,\bbv_h) .
\end{equation}

We have

\begin{equation}\label{eq:stabest2}
\begin{aligned}
\A_h (\bfsigma_h & - \bfsigma_I,\bbu_h - \bbu_I ;\bftau_h,\bbv_h)=
\A_h (\bfsigma_h ,\bbu_h ;\bftau_h,\bbv_h) - \A_h ( \bfsigma_I,  \bbu_I;\bftau_h,\bbv_h)\\
& = - (\bbf,\bbv_h) - \A_h ( \bfsigma_I,  \bbu_I;\bftau_h,\bbv_h)\\
& = \A( \bfsigma,  \bbu;\bftau_h,\bbv_h)-\A_h ( \bfsigma_I,  \bbu_I;\bftau_h,\bbv_h)\\
& = \left[a(\bfsigma,\bftau_h) - a_h(\bfsigma_I,\bftau_h)\right]
 + \left(\bdiv\bftau_h, \bbu-\bbu_I \right) +\left( \bdiv(\bfsigma-\bfsigma_I), \bbv_h\right)\\
&= T_1+T_2+T_3
\end{aligned}
\end{equation}

Concerning $T_1$, it holds:

\begin{equation}\label{eq:stabest3}
\begin{aligned}
T_1 &= \sum_{E\in\Th} \left[ a_E(\bfsigma,\bftau_h) -a_E^h(\bfsigma_I, \bftau_h) \right]\\
&= \sum_{E\in\Th} \big[ a_E(\bfsigma,\bftau_h) -a_E(\Pi_E\bfsigma_I, \Pi_E\bftau_h) \\
&  \qquad  - \kappa_E\, h_E\int_{\partial E} \left[(Id-\Pi_E)\bfsigma_I\bbn\right]\cdot \left[(Id-\Pi_E)\bftau_h\bbn\right] \big]\\
&= \sum_{E\in\Th} \big[ a_E(\bfsigma-\bfsigma_\pi,\bftau_h) -a_E(\Pi_E( \bfsigma_I -\bfsigma_\pi), \Pi_E\bftau_h) \\
& \qquad  - \kappa_E\, h_E\int_{\partial E} \left[(Id-\Pi_E)\bfsigma_I\bbn\right]\cdot \left[(Id-\Pi_E)\bftau_h\bbn\right] \big] .
\end{aligned}
\end{equation}
We have, using the continuity of $a_E(\cdot,\cdot)$ and of $\Pi_E$:

\begin{equation}\label{eq:stabest4}
\begin{aligned}
\sum_{E\in\Th} & \big[ a_E(\bfsigma-\bfsigma_\pi,\bftau_h) -a_E(\Pi_E( \bfsigma_I -\bfsigma_\pi), \Pi_E\bftau_h) \big]\\
& \lesssim \left( ||\bfsigma-\bfsigma_\pi||_{0,\O} + ||\bfsigma_I-\bfsigma_\pi||_{0,\O}
\right) || \bftau_h ||_{0,\O} \\
& \lesssim \left( ||\bfsigma-\bfsigma_\pi||_{0,\O} + ||\bfsigma_I-\bfsigma ||_{0,\O} +
||\bfsigma-\bfsigma_\pi ||_{0,\O}\right) || \bftau_h ||_{0,\O}\\
& \lesssim \left( ||\bfsigma-\bfsigma_\pi||_{0,\O} + ||\bfsigma-\bfsigma_I ||_{0,\O} \right) || \bftau_h ||_\Sigma.
\end{aligned}
\end{equation}
Furthermore, it holds:

\begin{equation}\label{eq:stabest5-0}
\begin{aligned}
\sum_{E\in\Th} &   \kappa_E\, h_E\int_{\partial E} \left[(Id-\Pi_E)\bfsigma_I\bbn\right]\cdot \left[(Id-\Pi_E)\bftau_h\bbn\right] \big] \\
& \lesssim \sum_{E\in\Th}   h_E^{1/2} ||(Id-\Pi_E)\bfsigma_I\bbn||_{0,\partial E}
 h_E^{1/2} ||(Id-\Pi_E)\bftau_h\bbn||_{0,\partial E}
\end{aligned}
\end{equation}
Under assumptions $\mathbf{(A1)}$ and $\mathbf{(A2)}$, we notice that, given $\bftau_h\in\Sigma_h(E)$, we have the 1D inverse estimate on $\partial E$:

\begin{equation}\label{eq:n-trace-1}
h_E^{1/2} ||\bftau_h\bbn||_{0,\partial E}\lesssim
||\bftau_h\bbn||_{-1/2,\partial E}
\qquad \forall \bftau_h\in\Sigma_h(E) .
\end{equation}
Using the techniques developed in \cite{BLRXX}, we deduce the scaled trace estimate:

\begin{equation}\label{eq:n-trace-2}
||\bftau_h\bbn||_{-1/2,\partial E}
\lesssim
||\bftau_h||_{0,E} + h_E || \bdiv\bftau_h ||_{0,E}
\qquad \forall \bftau_h\in\Sigma_h(E) .
\end{equation} 
Hence, we get:

\begin{equation}\label{eq:scaling}
h_E^{1/2} ||\bftau_h\bbn||_{0,\partial E}\lesssim
||\bftau_h||_{0,E} + h_E || \bdiv\bftau_h ||_{0,E}
\qquad \forall \bftau_h\in\Sigma_h(E) .
\end{equation}
From \eqref{eq:stabest5-0} and \eqref{eq:scaling} we then deduce

\begin{equation}\label{eq:stabest5-1}
\begin{aligned}
\sum_{E\in\Th} &   \kappa_E\, h_E\int_{\partial E} \left[(Id-\Pi_E)\bfsigma_I\bbn\right]\cdot \left[(Id-\Pi_E)\bftau_h\bbn\right] \big] \\
& \lesssim \left( \sum_{E\in\Th} \Big(  ||(Id-\Pi_E)\bfsigma_I||^2_{0,E} +
h_E^2||\bdiv\bfsigma_I||^2_{0,E}\Big)
\right)^{1/2}
||\bftau_h||_{\Sigma} .
\end{aligned}
\end{equation}

Since it holds, using also the $L^2$ continuity of $\Pi_E$:

\begin{equation}\label{eq:stabest5-2}
\begin{aligned}
||(Id-\Pi_E)\bfsigma_I||^2_{0,E}& =
||(\bfsigma_I - \bfsigma_\pi) + \Pi_E(\bfsigma_\pi - \bfsigma_I) ||^2_{0,E}\\
& \lesssim ||\bfsigma_I - \bfsigma_\pi ||^2_{0,E}
\lesssim ||\bfsigma_I - \bfsigma ||^2_{0,E} + ||\bfsigma - \bfsigma_\pi ||^2_{0,E} .
\end{aligned}
\end{equation}

Therefore, we get:

\begin{equation}\label{eq:stabest5}
\begin{aligned}
\sum_{E\in\Th} &   \kappa_E\, h_E\int_{\partial E} \left[(Id-\Pi_E)\bfsigma_I\bbn\right]\cdot \left[(Id-\Pi_E)\bftau_h\bbn\right] \big] \\
& \lesssim \left( ||\bfsigma_I - \bfsigma ||_{0,\O} + ||\bfsigma - \bfsigma_\pi ||_{0,\O} + h\,||\bdiv\bfsigma_I||_{0,\O}
\right)
||\bftau_h||_{\Sigma} .
\end{aligned}
\end{equation}

Combining \eqref{eq:stabest3}, \eqref{eq:stabest4} and \eqref{eq:stabest5}, we infer

\begin{equation}\label{eq:stabest6}
T_1\lesssim \left(   ||\bfsigma- \bfsigma_I||_{0,\O}+ ||\bfsigma - \bfsigma_\pi)||_{0,\O} +h\,||\bdiv\bfsigma_I||_{0,\O}
\right)  ||\bftau_h||_\Sigma .
\end{equation}

Regarding $T_2$, $T_3$ and $T_4$, one obviously have:

\begin{equation}\label{eq:stabest7}
\left\lbrace{
\begin{aligned}	
&T_2\lesssim ||\bbu -\bbu_I||_U ||\bftau_h||_\Sigma \\
&T_3\lesssim ||\bfsigma- \bfsigma_I||_\Sigma ||\bbv_h||_U .
\end{aligned}
} \right. .
\end{equation}

From \eqref{eq:stabest}, \eqref{eq:stabest2}, \eqref{eq:stabest6} and \eqref{eq:stabest7}, we get:

\begin{equation}\label{eq:stabest8}
\begin{aligned}
||\bfsigma_h - \bfsigma_I||_\Sigma  + || \bbu_h &- \bbu_I ||_U
\lesssim  \Big( ||\bfsigma- \bfsigma_I||_\Sigma+ ||\bfsigma - \bfsigma_\pi)||_{0,\O} \\
&+ h\,||\bdiv\bfsigma_I||_{0,\O}+||\bbu -\bbu_I||_U\Big)
\left(  || \bftau_h||_\Sigma + || \bbv_h ||_U \right)
\end{aligned}
\end{equation}

Estimate \eqref{eq:erroreq} follows from the triangle inequality, estimate \eqref{eq:stabest8} and bound \eqref{eq:costest}.
\end{proof}

We are now ready to state and prove our main convergence result.

\begin{thm}\label{th:main_convergence}
Let $(\bfsigma,\bbu)\in\Sigma\times U$ be the solution of Problem \eqref{cont-pbl}, and let $(\bfsigma_h,bbu_h)\in\Sigma_h\times U_h$ be the solution of the discrete problem \eqref{eq:discr-pbl-ls}. Suppose that  assumptions $\mathbf{(A1)}$ and $\mathbf{(A2)}$ are fulfilled.
Assuming $\bfsigma_{|E}\in H^1(E)^{2\times 2}_s$ and $(\bdiv\, \bfsigma)_{|E}\in H^1(E)^2$, the following estimate holds true:

\begin{equation}\label{eq:main_conv-est}
|| \bfsigma - \bfsigma_h||_{\Sigma} + || \bbu - \bbu_h||_U \lesssim C(\Omega,\bfsigma,\bbu)\, h ,
\end{equation}
where $C(\Omega,\bfsigma,\bbu)$ is independent of $h$ but depends on the domain $\Omega$ and on the Sobolev regularity of $\bfsigma$ and $\bbu$.
\end{thm}

\begin{proof}
In Proposition \ref{pr:error-est} let us choose $\bfsigma_I=\Ih\bfsigma \in\Sigma_h$ as detailed in Section \ref{ss:interpol-oper}, $\bbu_I=P_0\bbu\in U_h$ and $\bfsigma_\pi= P_0\bfsigma\in \Sigma_h$. Estimate \eqref{eq:main_conv-est} easily follows from Proposition \ref{pr:approxest} and standard approximation results.
\end{proof}

\begin{remark}
An alternative way to develop the stability and error analysis might be the use of suitable mesh-dependent norms, as detailed in \cite{lovadina1} for the Poisson problem in mixed form.
\end{remark}

\section{Conclusions}\label{s:conclusions}

We have proposed, numerically tested and analysed a new Virtual Element Method for the Hellinger-Reissner formulation  of two-dimensional elasticity problems. Our scheme is low-order, it has a-priori symmetric stresses and it optimally converges. Possible future developments of the present study include the design of higher-order schemes in the framework of the same variational principle. In addition, accurate post-processed displacements might be considered and used for mesh adaptive strategies, based on suitable a-posteriori error estimators.

\medskip

\begin{center}
{\large {\bf Aknowledgements}}
\end{center}

EA gratefully acknowledges the partial financial support of the Italian Minister of University and Research, MIUR  (Program: Consolidate the Foundations 2015; Project: BIOART; Grant number (CUP): E82F16000850005).
%
%

\medskip


\bibliographystyle{amsplain}

\bibliography{general-bibliography,biblio,VEM}

 \end{document}